\documentclass [11pt]{article}

\usepackage{amsfonts,amsthm,amsmath,amssymb,amscd}
\usepackage{psfrag,graphicx,subfigure}
\usepackage{import} 

\pdfoutput=1

\newcommand{\bn}{{\mathbf N}}
\newcommand{\bc}{{\mathbb C}}
\newcommand{\br}{{\mathbb R}}
\newcommand{\bs}{{\mathbb S}}
\newcommand{\bz}{{\mathbb Z}}
\newcommand{\bd}{{\mathbb D}}

\newcommand{\bbC}{{\mathbb C}}

\newcommand{\rat}{{\mathcal M}_2}
\newcommand{\calh}{{\mathcal H}}
\newcommand{\calo}{{\mathcal O}}
\newcommand{\per}{\rm Per}
\newcommand{\J}{\mathcal J(f)}
\newcommand{\eps}{\varepsilon}
\renewcommand{\hat}{\widehat}
\newcommand{\orb}{\calo}

\newtheorem{Theorem}{Theorem}[section]
\newtheorem{Corollary}[Theorem]{Corollary}
\newtheorem{Lemma}[Theorem]{Lemma}
\newtheorem{Proposition}[Theorem]{Proposition}
\newtheorem*{thmA}{Theorem A}
\newtheorem*{thmB}{Theorem B}
\newtheorem*{thmC}{Theorem C}

\theoremstyle{definition}
\newtheorem{Remark}[Theorem]{Remark}
\newtheorem{Definition}[Theorem]{Definition}

\newsavebox{\savepar}

\usepackage{color}
\definecolor{green}{rgb}{0,0.7,0.3}
\definecolor{red}{rgb}{0.9,0,0.5}

\title{Sierpi\'{n}ski curve Julia sets for quadratic rational maps}

\author{\small Robert L.~Devaney \\
{\small Department of Mathematics}\\
{\small Boston University}\\
{\small 111 Cummington Street}\\
{\small Boston, MA 02215, USA} \and 
{\small N\'uria Fagella }\\
{\small Dept.~de Matem\`atica Aplicada i An\`alisi}\\
{\small Universitat de Barcelona}\\
{\small Gran Via de les Corts Catalanes, 585}\\
{\small 08007 Barcelona, Spain} \and 
{\small Antonio Garijo}\\
{\small Dept.~d'Eng.~Inform\`atica i Matem\`atiques}\\
{\small Universitat Rovira i Virgili}\\
{\small Av. Pa\"isos Catalans 26}\\
{\small Tarragona 43007, Spain} \and 
{\small Xavier Jarque \thanks{ corresponding autor:   xavier.jarque@ub.edu}}\\
{\small Dept.~de Matem\`atica Aplicada i An\`alisi}\\
{\small Universitat de Barcelona}\\
{\small Gran Via de les Corts Catalanes, 585}\\
{\small 08007 Barcelona, Spain} }

\date{\today}

\begin{document}
\maketitle

\abstract{ We investigate under which dynamical conditions the Julia set of a quadratic rational map  is a Sierpi\'{n}ski curve. }

\section{Introduction}\label{section:intro}

Iteration of rational maps in one complex variable has been widely studied in recent decades continuing the remarkable papers of P. Fatou and G. Julia who introduced normal families and Montel's Theorem to the subject at the beginning of the twentieth century. Indeed, these maps are  the natural family of functions when considering iteration of holomorphic maps on the Riemann sphere $\mathbb {\hat C}$. For a given rational map $f$, the sphere splits into two complementary domains: the {\it Fatou set} $\mathcal F(f)$ where the family of iterates $\{f^n(z)\}_{n\geq 0}$ forms a normal family, and its complement,  the {\it Julia set} $\mathcal(f)$. The Fatou set, when non-empty, is given by the union of, possibly, infinitely many open sets in $\hat{\bbC}$, usually called Fatou components. On the other hand, it is  known that the Julia set is a closed, totally invariant, perfect non-empty set, and coincides with the closure of the set of (repelling) periodic points. For background see \cite{milnor}. 

Unless the Julia set of $f$  fills up the whole sphere, one of the major problems in complex dynamics is to characterize the topology of the Julia set (or at least determine some topological properties) and, if possible, study the chaotic dynamics on this invariant set when iterating the map.  Indeed, depending on $f$, the Julia set can have either trivial topology (for instance just a circle), or a highly rich topology (for instance it may be a  non locally connected continuum, a dendrite, a Cantor set, a Cantor set of circles, etc.) 

The Sierpi\'{n}ski carpet fractal shown in Figure \ref{fig:carpet} is one of the best known planar, compact and connected sets. On the one hand, it is a {\it universal plane continuum} in the sense that it contains a homeomorphic copy of any planar, one-dimensional, compact and connected set. On the other hand, there is a topological characterization of this set due to G. Whyburn \cite{Why} 
which explain why it is not uncommon to find Sierpi\'{n}ski carpet like-sets in complex dynamics

\begin{Theorem}[\cite{Why}, Theorem 3]\label{theorem:why}
Any non-empty planar set that is compact, connected, locally connected, nowhere dense, and has the property that any two complementary domains are bounded by disjoint simple closed curves is homeomorphic to the Sierpi\'{n}ski carpet.
\end{Theorem}

Sets with this property are known as  {\it Sierpi\'{n}ski curves}. Building bridges among complex dynamics and {\it Sierpi\'{n}ski curves}  is the main goal of different studies including this paper. The first example of a (rational) map whose Julia set is a Sierpi\'{n}ski curve is due to  J.~Milnor and L.~Tan  (\cite{MilLei}) in 1992. Their example belongs to   
the family of quadratic rational maps given by $z \mapsto a(z+1/z)+b$. Almost at the same time, in his thesis, K. Pilgrim gave the cubic, critically finite, family of rational maps $z\to c(z-1)^2(z+2)/(3z-2)$ having Sierpi\'{n}ski curve Julia sets for some values of $c$ (for instance $c\approx 0.6956$). Unlike to J. Milnor and T. Lei, who proved their result using polynomial-like maps, K. Pilgrim proved the existence of Sierpis\'nki curve Julia sets  from a systematic study of the contacts among boundaries of Fatou components. 

\begin{figure}[hbt!]
	\centering
	\includegraphics[width=0.5\textwidth]{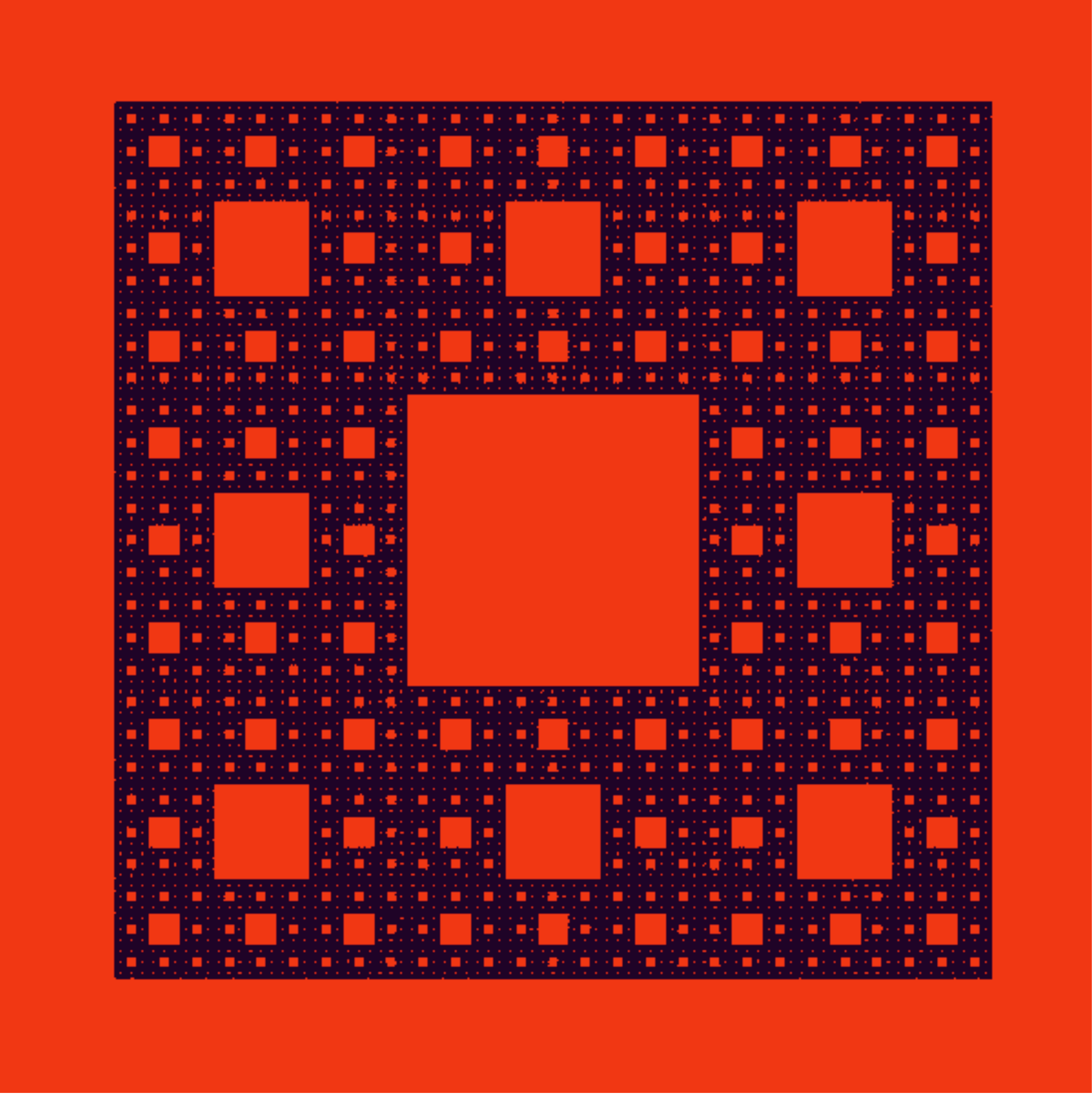}
	\parbox{0.85\textwidth}{\caption{\small The Sierpi\'{n}ski carpet fractal. The black region corresponds the the limit set by taking out the corresponding central  squares.}}
		\label{fig:carpet}
\end{figure}

More recently, other authors have  shown that the Julia sets of  a rational map of
arbitrary degree can be a Sierpi\'{n}ski curve
(\cite{escapet,Stein}). For example, in \cite{escapet},
Sierpi\'{n}ski curve Julia sets were shown to occur in the family  $z \mapsto
z^n + \lambda/z^d$  for some values of  $\lambda$, and, in  \cite{Stein}, for
the rational map $z \mapsto t (1+(4/27)z^3/(1-z))$ also for some values of
$t$.  However, it is not only rational maps that can exhibit  Sierpi\'{n}ski
curve  Julia sets,  as was proven by S. Morosawa in \cite{Moro}. He showed that the  
entire transcendental family $z \mapsto a e^a (z-(1-a)) e^z,$ have Sierpi\'{n}ski curve Julia sets for all $a>1$. Notice that, 
for those maps, the Julia set includes a non-locally connected Cantor bouquet (Cantor set of curves) making this result highly unexpected (see also \cite{GarJarMor} for more details). In Figure
\ref{fig:examples} we show four examples of Sierpi\'{n}ski curve Julia sets, one in each of the  families mentioned above.

\begin{figure}[ht]
    \centering
    \subfigure[\scriptsize{Milnor and Tan Lei's example $ -0.138115091(z+1/z) - 0.303108805 $.}  ]{
     \includegraphics[width=0.4\textwidth]{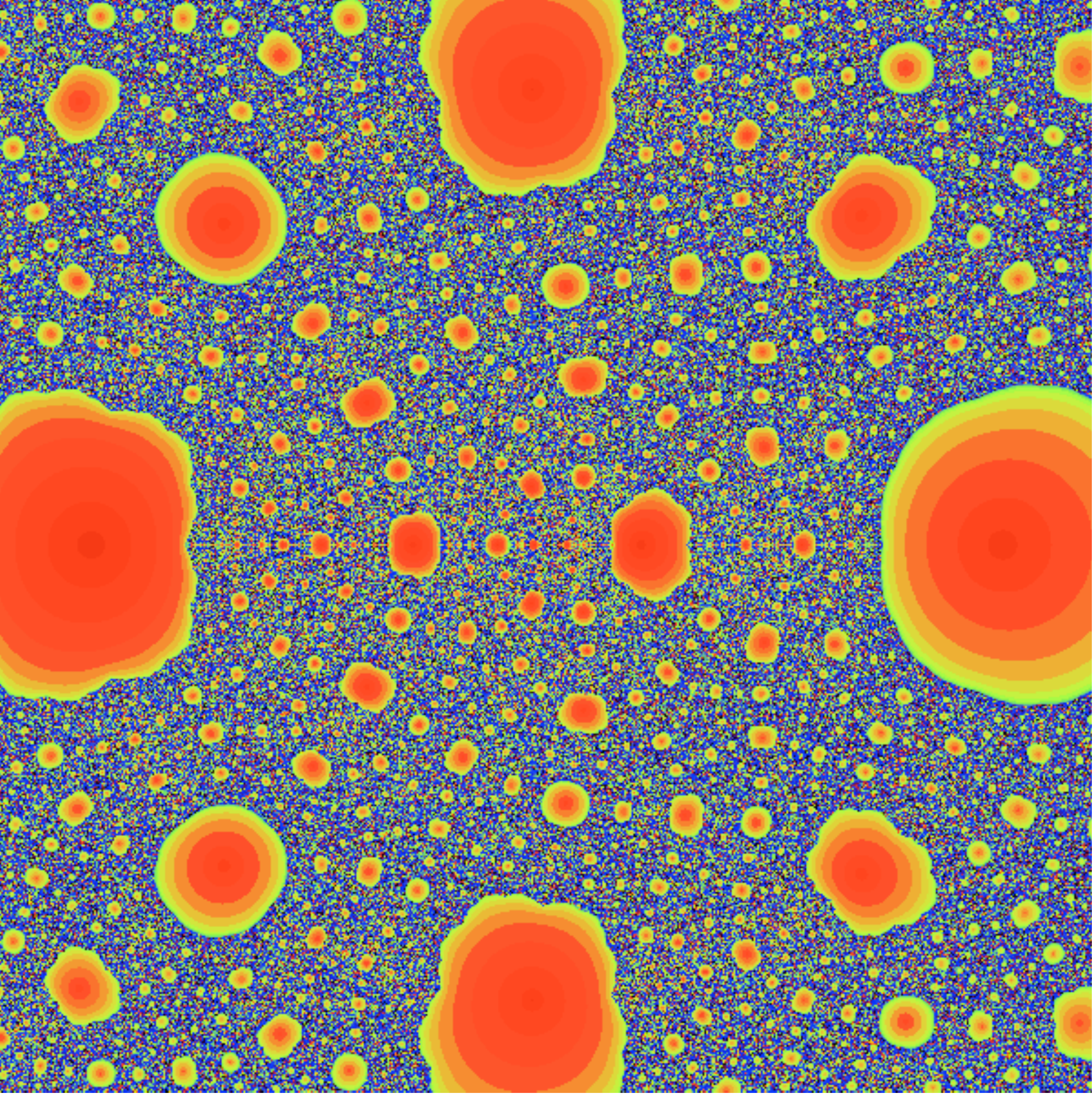}}
    \subfigure[\scriptsize{Devaney's example  $z^2-\frac{1}{16 z^2} $.}  ]{
    \includegraphics[width=0.4\textwidth]{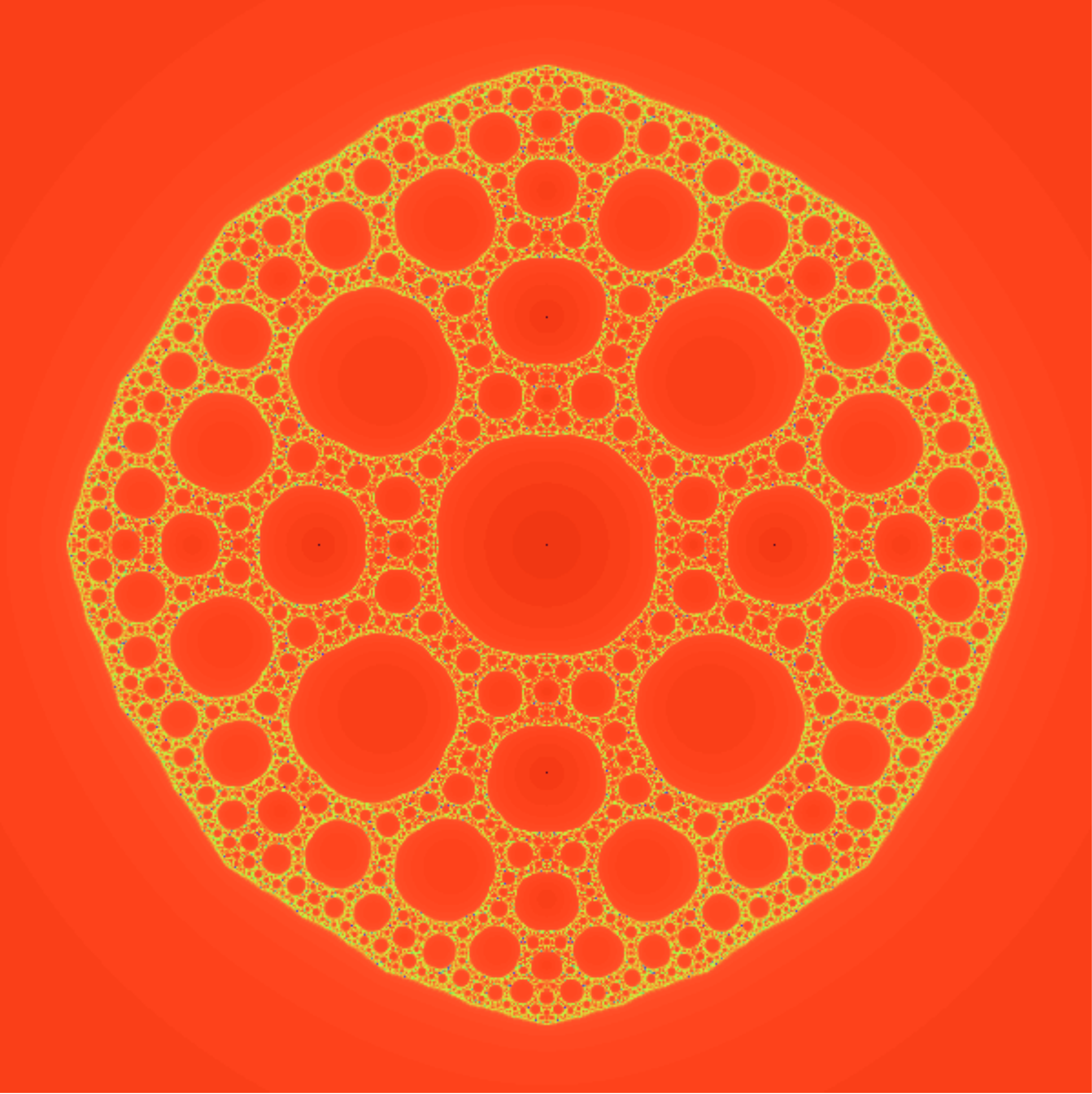}}
    \subfigure[\scriptsize{Steinmetz's example  $1+(4/27)z^3/(1-z)$.}  ]{
     \includegraphics[width=0.4\textwidth]{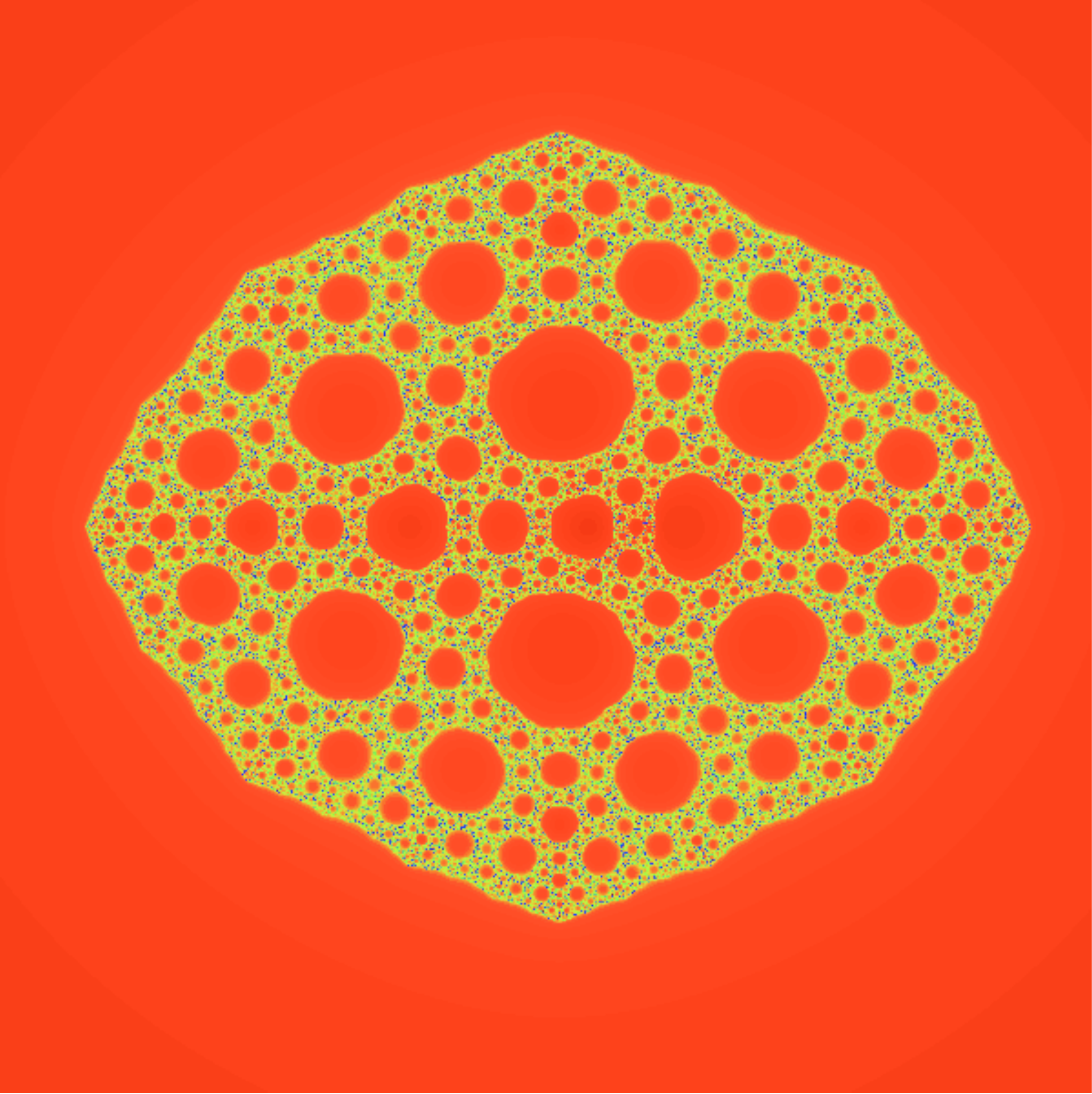}}
    \subfigure[\scriptsize{Morosawa's example $\quad 1.1 (e^z(z-1)+1)$.}  ]{
     \includegraphics[width=0.4\textwidth]{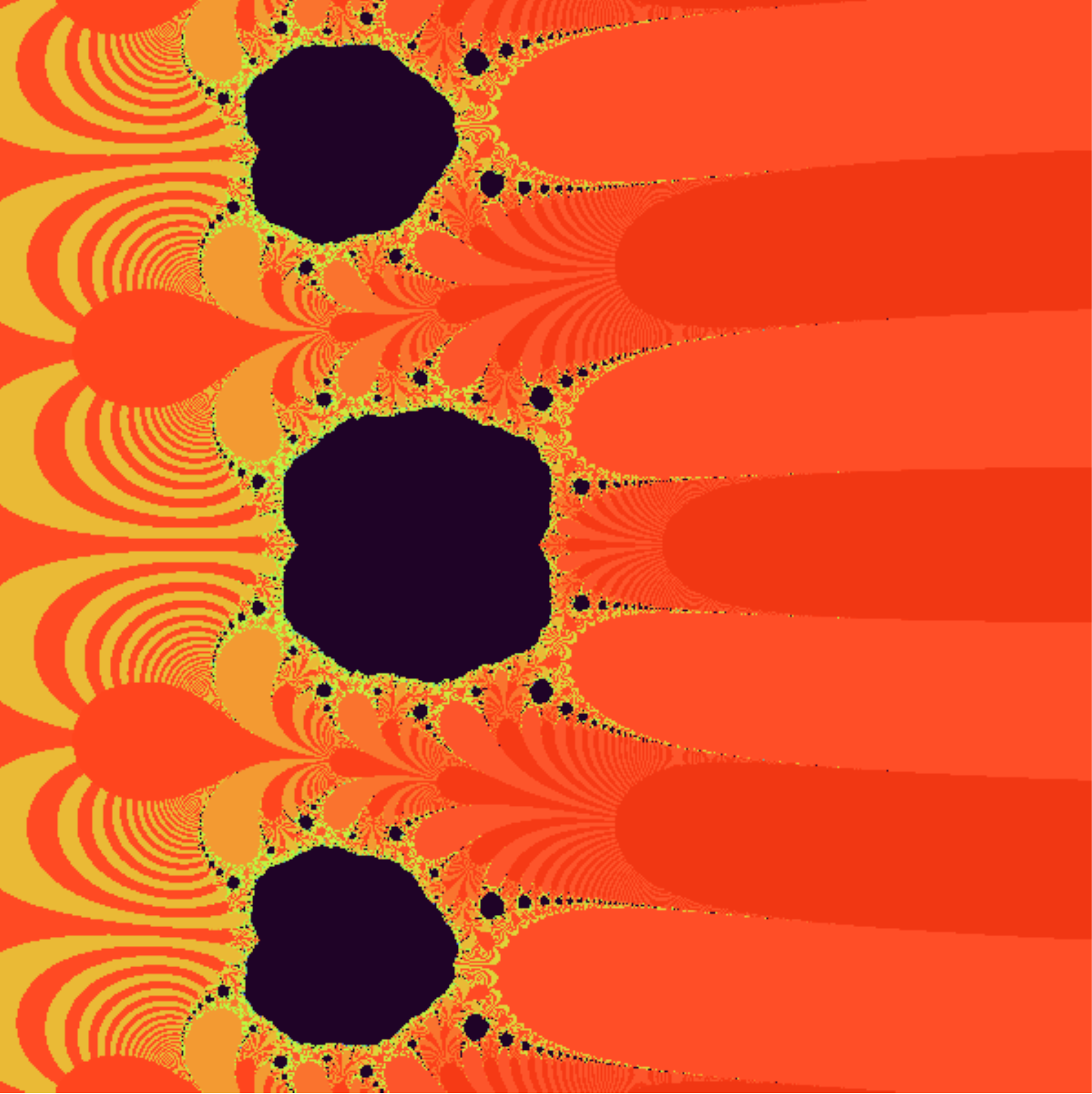}}
    \caption{\small{Two examples of Sierpi\'{n}ski curve Julia sets.}}    
    \label{fig:examples}
    \end{figure}

In this paper we present a more systematic approach to the
problem of existence of Sierpi\'{n}ski curves as Julia sets of rational
maps. In most of the cases mentioned above, the functions at hand have a
basin of attraction of a superattracting periodic orbit, which additionaly captures 
all of the existing critical points. Our goal is to find sufficient  and, if possible, also necessary dynamical conditions under which we can assure that 
the Julia set of a certain rational map is a Sierpi\'{n}ski curve. 

To find  general conditions for all rational maps is a long term
program. In this paper we restrict to rational maps of degree two (quadratic rational maps in what follows) which have an attracting periodic orbit, i.e., those which belong to
$\per_n(\lambda)$ for some $|\lambda|<1$, the multiplier of the attracting periodic
orbit of period $n$. We cannot even characterize all of those, but we cover mainly all of the
hyperbolic cases. To do so, we take advantage of the work of M. Rees \cite {Re0,Re1,Re2}, J. Milnor \cite{MilLei} and K. Pilgrim \cite{Pil} who  deeply studied quadratic rational maps and its parameter space. Indded, the space of all  quadratic rational maps  from the Riemann sphere to itself can be parametrized using 5 complex parameters. However, the space consisting of all conformal conjugacy classes is biholomorphic to  $\bc^2$ \cite{Mil} and will be denoted by $\rat$.

Following \cite{Re0}, hyperbolic  maps in $\rat$ can be classified into four
types {\it A, B, C and D,}  according to the behaviour of their two critical
points:  {\it Adjacent} (type $A$), {\it Bitransitive}  (type $B$), {\it
  Capture}  (type $C$) and {\it Disjoint}  (type D). In the {\it Adjacent}
type, both critical points belong to the same Fatou component; in the {\it Bitransitive} case  the critical points belong to two different Fatou components, both part of the same immediate basin of attraction; in the {\it Capture} type  both critical points belong to the basin of an attracting periodic point but only one of them belongs to the immediate basin; and  finally, in the  {\it Disjoint} type,  the two critical points belong to the attracting basin of two disjoint attracting cycles. 

\begin{figure}[hbt!]
	\centering
	\includegraphics[width=0.5\textwidth]{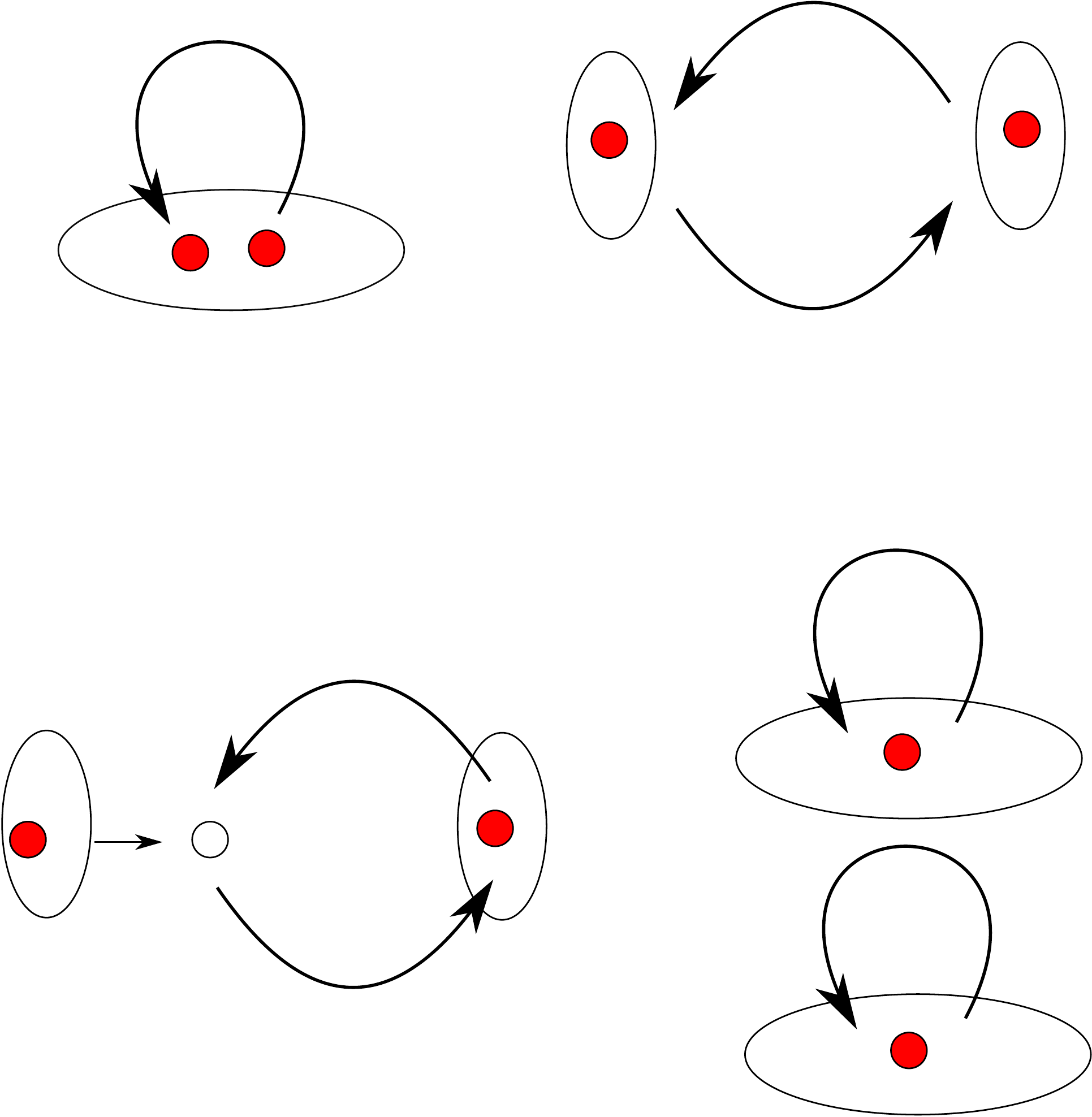}
	\parbox{0.85\textwidth}{\caption{\small Sketch of the different types of hyperbolic maps attending to the behaviour of the critical orbits.}}
		\label{fig:hyper_types}
\end{figure}

In many of our statements we consider one-dimensional complex slices of $\rat$ and in particular to  $\per_n(0)$, for $n\geq 1$. These slices  $\per_n(0)$ contain all the conformal  conjugacy classes of maps with a periodic critical orbit of period $n$. The first slice, $\per_1(0)$, consists of all  quadratic  rational maps having a fixed critical point, which must be superattracting. By sending this point to infinity and the other critical point to 0, we see that all rational maps in this slice are conformally conjugate to a quadratic polynomial of the form  $Q_c(z)=z^2+c$. Consequently, there are no Sierpi\'{n}ski curve Julia sets in $\per_1(0)$, since any Fatou component must share boundary points with the basin of infinity. The second slice, $\per_2(0)$, consists of all quadratic rational maps having a period two critical orbit. Such quadratic rational maps has been investigated  by M. Aspenberg and M. Yampolsky \cite{matings}, where the authors consider the mating between the Basilica with other suitable quadratic polynomials. Among other results they proved that the two Fatou components containing  the period two critical orbit have non-empty intersection. Therefore there are no  Sierpi\'{n}ski curve Julia sets in $\per_2(0)$.  Hence Sierpi\'{n}ski carpets can only appear as Julia sets of maps in  $\per_n(0)$, for $n \geq 3$.

\begin{figure}[hbt]
    \centering
    \subfigure[\scriptsize{Parameter plane of $z^2+c$}  ]{
     \includegraphics[width=0.4\textwidth]{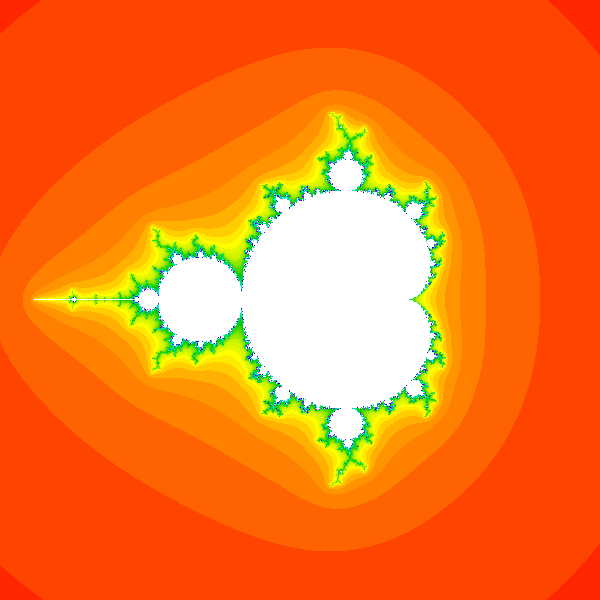}}
    \subfigure[\scriptsize{Parameter plane of $\frac{az-a^2/2}{z^2}$}  ]{
    \includegraphics[width=0.4\textwidth]{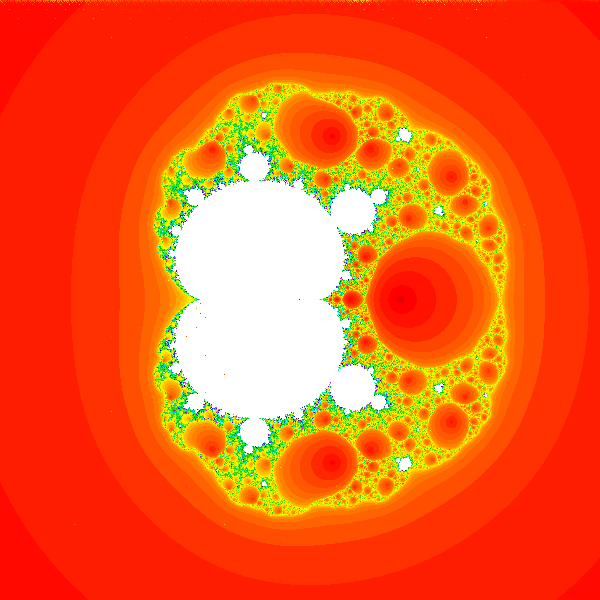}}
    \subfigure[\scriptsize{Parameter plane of   $\frac{(z-1)(z-a/(2-a))}{z^2}$}  ]{
     \includegraphics[width=0.4\textwidth]{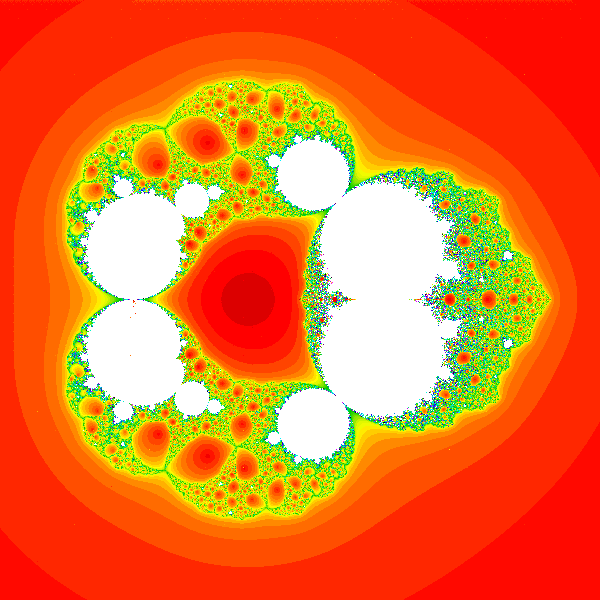}}
    \subfigure[\scriptsize{Parameter plane of  $\frac{(z-a)(z-(2a-1)/(a-1))}{z^2}$}  ]{
     \includegraphics[width=0.4\textwidth]{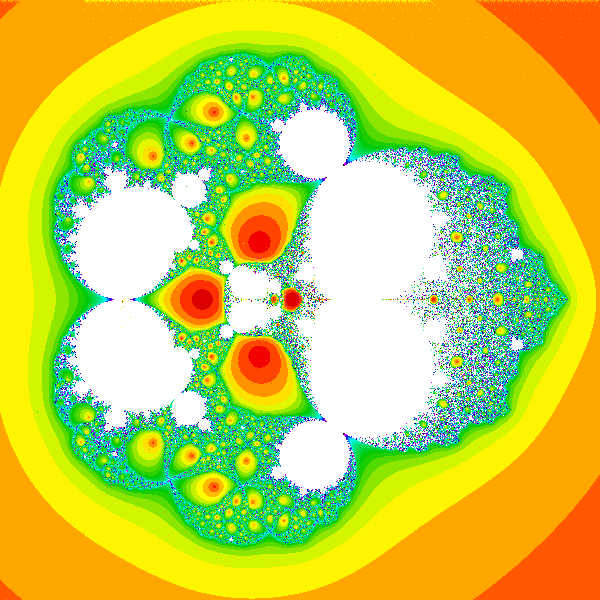}}
    \caption{\small{The slices $\per_1(0), \per_2(0), \per_3(0)$ and $\per_4(0)$}}    
    \label{fig:slices}
\end{figure}

In the hyperbolic setting, when dealing with the topology of the Julia set, restricting  to $\per_n(0)$ is not a  loss of generality. 
Indeed, if $f$ is a hyperbolic rational map of degree two not of type $A$ (we will see later that this is not a relevant restriction), it follows from Rees's Theorem (see Theorem \ref{th:Rees}) that the hyperbolic component $\calh$ which contains $f$ has a unique {\em center} $f_0$, i.e., a map for which all attracting cycles are actually superattracting. In other words, $\calh$ must intersect $\per_n(0)$ for some $n\geq 1$, and this intersection is actually a topological disc. Moreover,  all maps in $\calh$ are conjugate to $f_0$ in a neighborhood of their Julia set (see \cite{mss}). Hence the Julia set of $f_0 \in \per_n(0)$ is a Sierpi\'{n}ski curve if and only if the Julia set of all maps $f\in\calh$ are  Sierpi\'{n}ski curves. This discussion applies in particular, to maps in $\per_n(\lambda)$ with $|\lambda|<1$ of any type $B, C$ and $D$.

We now introduce some terminology in order to state our main results. Let $\lambda, \mu \in \bd$ and $n,m \in \bn$ with $n,m\geq3$. 
Suppose $f\in \per_n(\lambda)$. We denote  by $U$ the immediate basin of attraction of the attracting cycle and  $U_0, U_1, \cdots U_{n-1}$ the Fatou components which form the immediate basin of the attracting cycle. If $f \in \per_n(\lambda) \cap \per_m(\mu)$  then we denote by $U$ and $V$ the immediate basin of attraction of the two attracting cycles, and we denote by $U_0, U_1, \cdots U_{n-1}$ and $V_0, V_1, \cdots V_{m-1}$ the corresponding Fatou components.

As we mentioned before, our goal in this paper is to obtain {\it dynamical} conditions that ensure that the Julia set of a quadratic rational map is a  Sierpi\'{n}ski curve. The first requirement for a quadratic rational map to have a Sierpi\'{n}ski curve Julia set is that the map is hyperbolic. Using the hyperbolicity of the map and previous results of other authors the problem reduces to study the  contact between Fatou components. 

\begin{thmA} Let $n \geq 3$, and let $f \in \rat$ be a  hyperbolic map in $\per_n(\lambda)$ without (super) attracting fixed points. The following conditions hold.
\begin{enumerate}

\item[(a)] If $f$ is  of type $C$ or $D$ and $i \neq m$,  then $\partial U_{i}  \cap \partial U_m$ is either empty or reduces to a unique point $p$ satisfying $f^j(p)=p$, for some $1\leq j<n$ a divisor of $n$. 
\item[(b)]  Let $f$ be of type $D$, and  $f \in \per_n(\lambda) \cap \per_m(\mu)$  such that $gcd(n,m)=1$. Assume that  $\partial U_{i_1}  \cap \partial U_{i_2}= \emptyset$ for $0\leq i_1 < i_2 \leq n-1$ and $\partial V_{j_1}  \cap \partial V_{j_2}= \emptyset$ for $0 \leq j_1 < j_2 \leq m-1$.  Then $\partial U_i \cap \partial V_j = \emptyset$, for $0 \leq i \leq n-1$ and $0 \leq j \leq m-1$. 
\end{enumerate}
\end{thmA} 

Now we apply the above result in order to investigate when a hyperbolic rational map has a Sierpi\'{n}ski curve Julia set. The first statement of Theorem B follows from Lemma 8.2  in \cite{MilLei} but we include it here for completeness. 

\begin{thmB}
Let $n \geq 1$ and let  $f \in \rat$. Assume that $f \in \per_n(\lambda)$ is a hyperbolic map. Then the following statements hold.
\begin{enumerate}
\item[(a)] If $f$ is of type $A$ (Adjacent) then $\J$  is not a Sierpi\'{n}ski curve.
\item[(b)]  If $f$ is  of type $B$ (Bitransitive) and $n=1,2,3,4$ then  $\J$ is not a Sierpi\'{n}ski curve.
\item[(c)] If  $f$ is  of type $C$ (Capture), $n\geq3$ and $\partial U$ does not contain any fixed point of $f^j$ for $j\mid n$ and $j<n$ then $\J$ is a Sierpi\'{n}ski curve.
\item[(d)] Suppose $f$ is  of type $D$ (Disjoint) and  $n,m\geq 3$ with $gcd(n,m)=1$. If   $\partial U$ does not contain any fixed point of $f^j$ for $j \mid n$, $j<n$ and  $\partial V$ does not contain any fixed point of $f^j$ for $j \mid m$, $j<m$,   then $\J$ is a Sierpi\'{n}ski curve.
\end{enumerate}
\end{thmB}

As an application of Theorems A and B we can make a fairly complete study of
$\per_3(0)$ (with its extensions mentioned above).  According to Rees
\cite{Re3} it is possible to partition the one-dimensional slice into five
pieces, each with different dynamics. 
In Figure \ref{fig:slice3_cero} we display this partition, which we shall explain in detail in Section \ref{sec:slice3}. Two and only two of the pieces, $B_1$ and $B_\infty$, are hyperbolic components of type $B$ (Bitransitive). The regions $\Omega_1$, $\Omega_2$ and $\Omega_3$  contain all hyperbolic components of type $C$ (Capture) and $D$ (Disjoint) and, of course, all non--hyperbolic parameters. We can prove the following. 

\begin{thmC} Let $f\in\per_3(0)$. Then,
\begin{enumerate} 
\item[(a)] If $a\in B_1 \cup B_\infty$ then $\mathcal J \left(f_a\right)$ is  not a Sierpi\'{n}ski curve. 
\item[(b)] If  $a\in \Omega_2 \cup \Omega_3$  then $\mathcal J\left(f_a\right)$ is  not a Sierpi\'{n}ski curve.
\item[(c)] If $a \in \Omega_1$  is a type C parameter,  then $\mathcal J\left(f_a\right)$ is  a Sierpi\'{n}ski curve. 
\item[(d)] If $a \in \Omega_1$ is a disjoint parameter and $\partial V$ does not contain a fixed point of $f^j$ for $j \mid m$, $j<m$ and $3 \nmid m$ then $\J$ is a Sierpi\'{n}ski curve.
\end{enumerate} 
\end{thmC}

\begin{Remark} 
As mentioned above, if $f$ is hyperbolic, these properties  extend to all maps in the hyperbolic component in $\mathcal{M}_2$ which $f$ belongs to.
\end{Remark}

\begin{Remark}
Theorem C illustrates that in fact  when $n$ is a prime number, the conditions of Theorems A  and B reduce to study 
the location of the three fixed points of $f$. So, for those values of $n$ a deep study in parameter space is plausible.    
\end{Remark}

\begin{figure}[ht]
	\centering
	\includegraphics[width=0.5\textwidth]{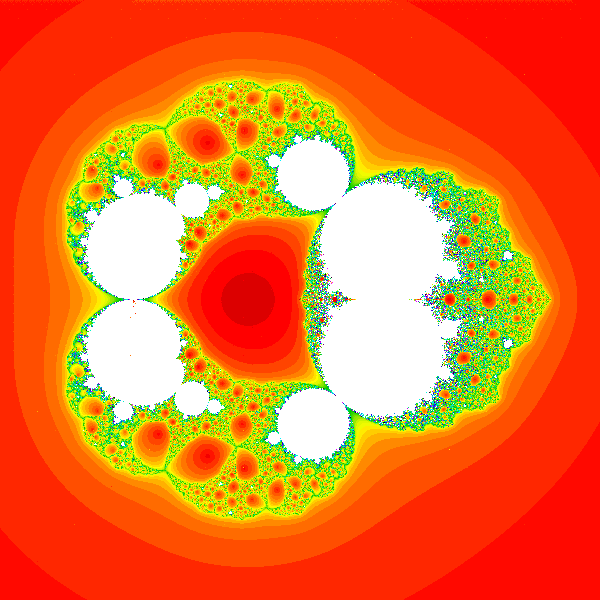}
	\setlength{\unitlength}{220pt}
	\put(-0.6,0.55){\small $B_1$}
	\put(-0.94,0.85){\small $B_\infty$}
	\put(-0.65,0.78){\small $\bf\Omega_2$}
	\put(-0.65,0.3){\small $\bf\Omega_3$}
	\put(-0.2,0.55){\small $\bf\Omega_1$}
	\caption{\small The slice $Per_3(0)$ and its pieces.}
	\label{fig:slice3_cero}
\end{figure}

The outline of the paper is as follows: in Section \ref{sec:previous} we give previous results concerning the topology of the Julia set of quadratic  hyperbolic rational maps.  In Section \ref{sec:sufficient} we concentrate on the contacts between boundaries of Fatou components.  In Section \ref{sec: proof_A_B} we prove Theorems A and B. Finally,
 in \ref{sec:slice3} we study the slice $Per_3(0)$ and prove Theorem C.

\bigskip
\noindent \emph{Acknowledgments.} We are grateful to the referee for many helpful comments and observations which helped us improve the paper. We also thank S\'ebastien Godillon for helpful discussions. The second, third and fourth authors are partially supported by the Catalan
grant 2009SGR-792, and by the Spanish grants MTM-2008-01486 Consolider
(including a FEDER contribution) and MTM2011-26995-C02-02. The first author was partially supported by grant \#208780 from the Simons Foundation. The second and fourth authors are also partially supported by Polish NCN grant decision DEC-2012/06/M/ST1/00168.

\section{Preliminary results} \label{sec:previous}

In this section we collect some results related to the topology of Julia sets of   rational maps, which we will use repeatedly. The first theorem  states a  dichotomy between the connectivity of the Julia set of a  quadratic  rational map and the dynamical behaviour of its critical points. 

\begin{Theorem}[\cite{MilLei}, Lemma 8.2]\label{th:Milnor}
The Julia set $\J$ of a quadratic rational map $f$ is either connected or
totally disconnected (in which case the map is conjugate on the Julia set to
the one-sided shift on two symbols). It is totally disconnected if and only if
either: 
\begin{enumerate}
\item[(a)] both critical orbits converge to a common attracting fixed point, or
\item[(b)] both critical orbits converge to a common parabolic fixed point of multiplicity two but neither critical orbit actually lands on this point.
\end{enumerate}
\end{Theorem}

\begin{Theorem}[\cite{milnor}, Theorem 19.2] \label{th:loco}
If the Julia set of a hyperbolic rational map is connected, then it is locally connected.
\end{Theorem}

The next  theorem, due to M. Rees, states that each hyperbolic component  of type $B$, $C$ and $D$ in the parameter space contains a critically finite rational map as its unique center. We also conclude that maps that belong to the same hyperbolic component are conjugate on their Julia set and so we will frequently consider only critically finite maps when referring to hyperbolic maps.

\begin{Theorem}[\cite{Re0}, Main Theorem, pp. 359-360]\label{th:Rees}
Let $\cal{H}$ be a hyperbolic component of type $B$,C or D of $\rat$. Then, $\cal{H}$ contains a unique center $f_0$, i.e., $f_0$ is the unique critically finite map inside the hyperbolic component $\cal{H}$. Moreover, all maps in the same hyperbolic component are J-stable.
\end{Theorem}

Another important result gives conditions under which we can assure that all Fatou components are Jordan domains. Recall that this was one of the conditions for having Sierpi\'{n}ski curve Julia sets.

\begin{Theorem}[\cite{Pil}, Theorem 1.1]\label{th:Pilgrim}
Let $f$ be a critically finite rational map with exactly two critical points, not counting  multiplicity. Then exactly one of the following possibilities holds:
\begin{enumerate}
\item[(a)] $f$ is conjugate to $z^d$ and the Julia set of $f$ is a Jordan curve, or
\item[(b)] $f$ is conjugate to a polynomial of the form $z^d+c, \, c \neq 0$, and the Fatou component corresponding to the basin of infinity under a conjugacy is 
the unique Fatou component which is not a Jordan domain, or 
\item[(c)] $f$ is not conjugate to a polynomial, and every Fatou component is a Jordan domain.
\end{enumerate} 
\end{Theorem}

We combine the two results above to get the following corollary.

\begin{Corollary} \label{cor:useful}
Let $f\in \rat$ be  hyperbolic or critically finite and assume $f$ has no (super) attracting fixed points. Then every Fatou component is a Jordan domain.
\end{Corollary} 
\begin{proof}
Since, by hypothesis, $f$ has no (super)attracting fixed points,  $f$ cannot be conjugate to a polynomial.

First assume that $f$ is critically finite, not necessarily hyperbolic. Then, using Pilgrim's Theorem \ref{th:Pilgrim} the corollary follows. If $f$ is hyperbolic,  it belongs to a hyperbolic component $\calh$. Let $f_0$ be its center, which exists and is unique by Rees's Theorem \ref{th:Rees}. Clearly, $f_0$ is critically finite and has no (super) attracting fixed points. Hence by Pilgrims's result all Fatou components of $f_0$ are Jordan domains. Since $f$ and $f_0$ belong to the same hyperbolic component, they are conjugate on a neighborhood of the Julia set and therefore $f$ has the same property.
\end{proof}

\section{Contact between boundaries of Fatou components: Proof of Theorems A and B} \label{sec:sufficient}

Throughout this section we assume that $f$ is a  hyperbolic quadratic  rational map  having a  (super)attracting period $n$ cycle with $n\geq 3$, or equivalently, $f$ is a hyperbolic map in $\per_n(\lambda)$ for some $n\geq 3$.  

A Sierpi\'{n}ski curve (Julia set) is any subset of the Riemann sphere homeomorphic to the Sierpi\'{n}ski carpet.  Consequently, due to Whyburn's Theorem (see the introduction), a Sierpi\'{n}ski curve Julia set  is  a Julia set which is compact, connected, locally connected, nowhere dense, and  such that any two complementary  Fatou domains are bounded by disjoint simple closed curves. The following lemma states that all but one of these properties are satisfied under the described hypotheses.

\begin{Lemma}\label{lemma:sierpins}
Let  $f\in \per_n(\lambda)$,  with $n\geq 3$, be hyperbolic. Then, the Julia set $\J$ is compact, connected, locally connected and nowhere dense.  Moreover, if $f$ has no (super) attracting fixed points (which is always the case for types B and C), then  $f$ is not of type A, and each Fatou component is a Jordan domain.
\end{Lemma}

\begin{proof}
The Julia set of a hyperbolic rational map is always a compact, nowhere dense subset of the Riemann sphere. If there are no (super) attracting fixed points, Theorem \ref{th:Milnor} implies  $\J$ is connected and hence locally connected (Theorem \ref{th:loco}). 

If $f$ is of type A without attracting fixed points, both critical points belong to the same (super) attracting Fatou component $U$ of period higher than 1. Since $\J$ is connected, $U$ is simply connected and therefore $f:U \to f(U)$ is of degree three ($U$ has two critical points) which is a contradiction since $f$ has global degree 2. 

Finally Corollary \ref{cor:useful} implies that all  Fatou components of $f$ are Jordan domains. 
\end{proof}

\begin{Remark}
In view of Theorem \ref{theorem:why} and Lemma \ref{lemma:sierpins}, if $f$ is a hyperbolic map in $\per_n(\lambda) \ n\geq 3$  without (super) attracting fixed points we have that $\J$ is a Sierpi\'{n}ski curve if and only their Fatou components have disjoint closure. 
\end{Remark}

To prove the main result of this section, Proposition \ref{prop:fixed_point_boundary}, we first establish some technical topological and combinatorial results that simplifies the exposition.

\begin{Lemma}\label{lem:three_domains_1}
Let $U,V,W$ be three disjoint planar Jordan domains and let $\gamma: \br / \bz \to \partial U$ be a parametrization of $\partial U$.  
\begin{enumerate}
\item [(a)] Let $a,b,c,d\in[0,1)$ be such that $0\leq a<c< 1, \, 0 \leq b<d<1$ and $\{a,b\} \cap \{c,d\} =\emptyset$. Assume that $ \gamma(a)$ and $ \gamma(c)$ belong to $\partial U \cap \partial V$ and $\gamma(b)$ and $\gamma(d)$ belong to $\partial U \cap \partial W$. Then, either $\{b,d\} \subset (a,c)$ or $\{b,d\} \subset \br/\bz \setminus (a,c)$.  
\item[(b)] Let $z_1,z_2, \cdots, z_{k},\ k\geq 1$ be $k$ different points in $\partial U \cap \partial V \cap \partial W$. Then $k \leq 2$.
\end{enumerate}
\end{Lemma}

\begin{proof} 
We first choose  three marked points $u,v$ and $w$ in $U, V$ and $W$, respectively. Since $U$ (respectively $V$ and $W$) is a Jordan domain, every boundary point  is accessible from the interior to the marked point $u$ (respectively $v$ and $w$) by a unique internal ray.  

First we prove statement (a). We build a (topological) quadrilateral
formed by two internal rays in $U$, joining $u$ and $\gamma(a)$ and another one  joining $u$ and $\gamma(c)$, and two internal rays in $V$ joining $v$ and $\gamma(a)$ and $v$ and $\gamma(c)$. This divides the Riemann sphere into two connected components $C_1$ and $C_2$ only one of which, say $C_1$, contains $W$. Thus $b$ and $d$ either both belong to the interval $(a,c)$ or both belong to the complement of $(a,c)$.

Second we prove statement (b).  Assume  $k \geq 3$. As before we build a (pseudo) quadrilateral formed by two internal rays in $U$ joining $u$ and $z_1$ and another one  joining $u$ and $z_2$ and two internal rays in $V$ joining $v$ and $z_1$ and $v$ and $z_2$. The complement of 
those rays (plus the landing points) in the Riemann sphere are two connected domains $C_1$ and $C_2$ only one of which, say $C_1$, contains $W$. We now add to the picture the two internal rays in $C_1$ connecting the point $w$ with $z_1$ and $z_2$, respectively. These new edges subdivide the domain $C_1$  into two domains, say $D_1$ and $D_2$. By construction the points $\{z_3 \cdots, z_{k}\} \in\partial U \cap \partial V \cap \partial W$ belong to one and only one of the domains $C_2$, $D_1$ or $D_2$. Therefore they cannot be accessed through internal rays by the three marked points $u,v$ and $w$, a contradiction. So $k\leq 2$. 
\end{proof}

\begin{Lemma}\label{lemma:three_dom}
Let $f$ be a rational map of degree $d\geq 2$. Let $U,V$ and $W$ be three different Jordan domains such that $f(U)=U$, $f(V)=V$ and $f(W)=W$. If there exists $p \in \partial U \cap \partial V \cap \partial W$,  then either $f(p)=p$ or $f'(p)=0$.
\end{Lemma}

\begin{proof}
By assumption the three Jordan domains $U$, $V$ and $W$ are Fatou components.
Let $p \in \partial U \cap \partial V \cap \partial W$ such that $f(p)\neq p$.  Notice that  $f(p) \in \partial U \cap \partial V \cap \partial W$ and  denote $\delta:=|f(p)-p|>0$. 

Take a circle $\gamma_\varepsilon$ around $p$ of radius $\varepsilon < \delta/3$. Since $f$ is holomorphic at $p$, if we assume $f'(p)\neq 0$ we can choose $\varepsilon$ small enough so that, if we go around $p$ counterclockwise once through $\gamma_\varepsilon$, then its image, $f(\gamma_\varepsilon)$, also gives one and only one turn  around $f(p)$  counterclockwise (in particular $f$ preserves orientation). Let $D_\varepsilon$ denote the disc of center $p$ and radius $\varepsilon$.

Denote by $u,v$ and $w$ three points in $U\cap \gamma_\varepsilon$, $V\cap \gamma_\varepsilon$ and $W\cap \gamma_\varepsilon$, respectively, which can be joined with $p$ through curves in $U\cap D_\eps$, $V\cap D_\eps$ and $W\cap D_\eps$ respectively.  Assume, without loss of generality, that starting at $u$, and turning counterclockwise, $\gamma_\varepsilon$ meets   $v$ and $w$ in this order. Let $\beta_U$ be a simple curve in $\overline{U}$ joining $p$,  $u$, $f(u)$ and $f(p)$. Similarly let $\beta_V$ be a simple curve in $\overline{V}$ joining $p$,  $v$, $f(v)$ and $f(p)$. Let $D_1$ and $D_2$ be the two connected components of $\hat\bc \setminus ( \beta_U \cup \beta_V)$. Choose  $D_1$ to be the region intersecting the arc of $\gamma_\eps$ going from $u$ to $v$ counterclockwise. Thus $D_1$ 
intersects the arc of $f\left(\gamma_\eps\right)$ going from $f(u)$ to $f(v)$ clockwise. It follows that  the points $w$ and $f(w)$ (and the whole domain $W$) should belong to $D_2$ which by construction intersects the piece of $f(\gamma_\eps)$ that goes from $f(u)$ to $f(v)$ counterclockwise. It is now immediate to see that one turn around $p$ implies two (or more) turns around $f(p)$, a contradiction with $f^{\prime}(p) \ne 0$. See Figure \ref{fig:one_two_turns}. Hence if $f'(p) \ne 0$ we should have $f(p)=p$ and the lemma follows.
\end{proof}

\begin{figure}[hbt!]
\def\svgwidth{0.8\textwidth}
\begin{center}
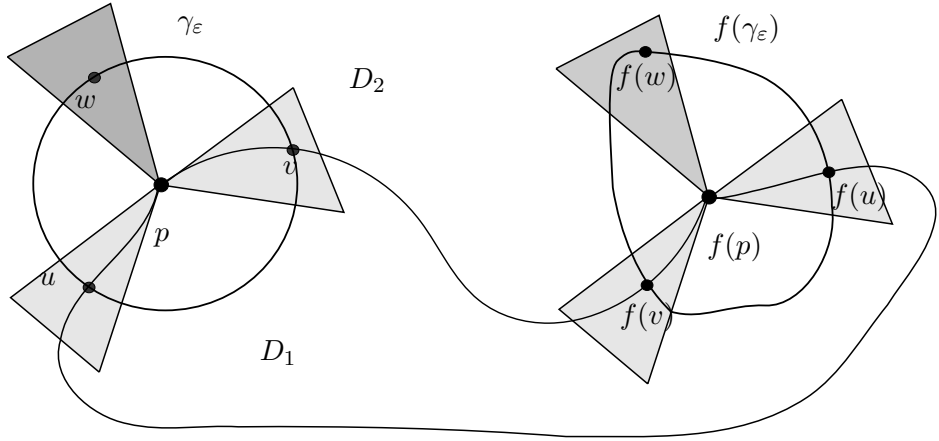
\end{center}
    \caption{\small{Sketch of the proof of Lemma \ref{lemma:three_dom}}}
\label{fig:one_two_turns}
\end{figure}

\begin{Remark}
The previous lemma only uses local properties of holomorphic maps. In particular it applies to rational maps of any degree.  In our case we will apply 
this lemma to a suitable iterate of a quadratic rational map.   
\end{Remark}

The above lemmas give some topological conditions on how the boundaries of the Fatou components may intersect. In what follows we will use frequently the fact that a certain map defined on the boundary of the Fatou components behaves like (more precisely it is conjugate to) the doubling map $\theta \to 2 \theta$ acting on the circle. The next lemma gives information on how the orbits of the doubling map distribute on the unit circle.

 \begin{Definition}
Let $\eta=\{\eta_0,\eta_1, \cdots, \eta_\ell, \cdots\}$ and $\tau=\{\tau_0,\tau_1, \cdots, \tau_\ell, \cdots\}$, with $\eta_i,\tau_i\in \br/\bz$ denote two different (finite or infinite) orbits under the doubling map. We say that $\eta$ and $\tau$ are {\em mixed} if  there exist four indexes $a,b,c$ and $d$ such that $\eta_a < \tau_b < \eta_c < \tau_d$  with the cyclic order of the circle.
\end{Definition}

\begin{Lemma}\label{lemma:notmix}
Consider the doubling map, $ \theta \to 2 \theta$, acting on the unit circle $\br/ \bz$. If $\tau$ and $\eta$ are two different orbits of the doubling map which are either finite and periodic of period $k\geq 3$ or infinite, then $\tau$ and $\eta$ are mixed. 

\end{Lemma}

\begin{proof}
Denote by $\Theta$ the doubling map. If the  binary expansion  of $\theta \in \br/\bz$ is $s(\theta)=s_0 s_1 s_2 \cdots ,\ s_j\in\{0,1\}$ then  $s_j=0$ if and only if $\Theta^j(\theta) \in [0,1/2)$. 
Consequently the four quadrants given by $(0,1/4), (1/4,1/2), (1/2,3/4), (3/4,1)$ correspond to  those angles whose binary expansion starts by  $00$, $01$, $10$ and $11$, respectively. Observe that the angles $0, 1/4, 1/2$ and $3/4$ are fixed or prefixed.

First suppose that the orbits are periodic of period $k \geq 3$. Then,  both cycles should have one angle in the second quadrant and one angle in the third quadrant. Moreover each orbit should have at least one angle  in the first quadrant (corresponding to two consecutive symbols 0), or one angle in the fourth quadrant (corresponding to two consecutive symbols 1).  Indeed, the only periodic orbit  touching neither the first nor the fourth quadrant is the unique 2-cycle $\{1/3,2/3\}$. 

Assume w.l.o.g. the $\eta$-cycle is the one having a point  in the third quadrant with largest argument (among the two cycles). Denote this point  by $\eta_c$.  Next we select one point of the $\tau$-cycle in the third quadrant, say $\tau_b$, and one point of the $\eta$-cycle in the second quadrant, say $\eta_a$. Finally we choose one point of the $\tau$-cycle belonging to either the fourth or the first quadrant; denoted by $\tau_d$. So by construction we have $\eta_a < \tau_b < \eta_c < \tau_d$, as we wanted to show.

The case of inifnite orbits follows similarly.
\end{proof}

We are ready to prove the main result of this section which implies Theorem A(a).

\begin{Proposition} \label{prop:fixed_point_boundary}
Let $n\geq 3$ and let $f \in \rat$ be a  hyperbolic map   of type $C$ or $D$ in $\per_n(\lambda)$ having no (super) 
attracting fixed points.  We denote by $U_0, U_1, \cdots, U_{n-1}$ the
Fatou components which form the immediate basin of attraction of an n-cycle. 
Then  for $i \neq m$, $\partial U_{i}  \cap \partial U_m$   is either empty or reduces to a 
unique point $p$ satisfying $f^{\ell}(p)=p$, for some $1\leq \ell<n$, dividing  $n$.
\end{Proposition}

\begin{proof}
Without loss of generality we  assume that $f$ is critically finite, i.e.,  $f \in \per_n(0)$ (see Theorem \ref{th:Rees}),  
$U_0$ contains the unique critical point belonging to the superattracting cycle under consideration,  and 
$f(U_i)=U_{i+1}$ (so, $f(u_i)=u_{i+1}$) (here and from now on, all indices are taken mod $n$). Moveover, since $f$ is of type $C$ or $D$ and $n \geq 3$, 
every  Fatou component is a Jordan domain (Corollary  \ref{cor:useful}) and the dynamics of 
$g:=f^n:\overline{U_i} \to \overline{U_i}$ is conformally conjugate to $z \to z^2$ on the closed unit disc 
$\overline{\bd}$ by the B\"ottcher map $\phi_i : \overline{U_i} \to \overline{\bd}$ (which is uniquely defined on 
each $U_i,\ i=0,\ldots, n-1$).  

The inverse of the B\"ottcher map  defines internal rays in every $\overline{U_i}$; more precisely 
$R_i(\theta)=\phi_i^{-1} \{ r e^{2 \pi i \theta}, 0 \leq r \leq 1 \}$ gives the internal ray in $U_i$ of angle $\theta \in \br / \bz$. 
The unique point of $R_i(\theta)$ in the  boundary of $U_i$ (that is, $\phi_i^{-1}(e^{2 \pi i \theta})$) 
will be denoted by $\hat{R}_i(\theta)$. 

The map $f$ induces the following dynamics on the internal rays 
$$\,f(R_0(\theta))=R_1(2 \theta), \, f(R_1(2\theta))=R_2(2\theta), \, \cdots, \,f(R_{n-1}(2 \theta))= R_0(2 \theta),$$ 
for every $\theta \in \br / \bz$. Similarly the equipotential $E_i(s)$ is defined by $\phi_i^{-1} \{ s e^{2 \pi i \theta}, 0 \leq \theta \leq 1 \}$ 
which cuts each internal ray once.

Assume there is a point $p \in\partial U_i \cap \partial U_m$ for $i \neq m$. By taking a suitable  iterate of $f$ we 
can assume, without loss of generality, that $p \in \partial U_0 \cap \partial U_j$, for some $0 < j<n$.  
Our goal is to show, by contradiction, that $g(p)=p$. Observe that, once this has been proved, we will have that  $f^j(p)=p$ for some $1\leq j\leq n$ a divisor of $n$. To exclude the case $j=n$, and conclude the statement of the proposition, we notice that this would imply that $g$ 
has a fixed point in each boundary of the $U_j$'s, a contradiction since $g|{\overline{U_j}},\ j=0,\ldots, n-1$ 
is conformally conjugate to  $z \to z^2$ on $\overline{\bd}$ and this map has $z=1$ as its unique fixed point on the 
unit circle.

We first show that the orbit of $p$ under $g$ cannot be pre-fixed, that is  
we cannot have $g(g^{\ell}(p))=g^{\ell}(p)$ for some $\ell >0$. Assume otherwise. Since the doubling map has a unique fixed point and a unique 
preimage of it (different from itself) we conclude that $g^{\ell}(p)=\hat{R}_0(0)$ and $g^{\ell-1}(p)=\hat{R}_0(1/2)$.  
Applying $f$ we have that $f(g^{\ell}(p))=\hat{R}_1(0)$ and 
$f(g^{\ell-1}(p))=f(\hat{R}_0(1/2))=\hat{R}_1(0)$ which implies that $f(g^{\ell}(p))=f(g^{\ell -1}(p))$. 
On the other hand we have that $f(g^{\ell}(p))=f(g^{\ell -1}(p)) \in \partial U_1 \cap \partial U_{j+1}$ 
and has two different preimages $g^{\ell}(p)$ and $g^{\ell-1}(p)$ in $\partial U_{j}$ while
$f:\overline{U}_j \to \overline{U}_{j+1}$ has degree one.
To deal with the finite periodic case or the infinite or non-preperiodic case we split the proof in two cases.

\vglue 0.1truecm
\noindent {\bf Case 1 ($2j\ne n$)} 
\vglue 0.1truecm
If the orbit of $p$ under $g$ is periodic of period $2$, that is $\{p,g(p)\} \in \partial U_0\cap \partial U_j$ with $p\ne g(p)$ and
$g^2(p)=p$, according to the previous notation, it
corresponds to the preimage by the B\"ottcher map of the periodic orbit $\{1/3,2/3\}$ under the doubling map (there is a unique periodic orbit of period two). Applying $f^j$ we have that  the orbit 
$\{f^j(p),f^j\left(g(p)\right)\} \in \partial U_j\cap \partial U_{2j}$ also corresponds to the preimage by the B\"ottcher map  of the same 
periodic orbit $\{1/3,2/3\}$ under the doubling map. Hence these 
two cycles (lying in $\partial U_j$) $\{ p, g(p)\}$ and $\{f^j(p),f^j(g(p))\}$, are the same cycle. We remark that we do not know if $p=f^j(p)$ or 
$p=f^j(g(p))$, we only claim that, as a cycle, it is the same one.
By construction we know that $p \in \partial U_0 \cap \partial U_j \cap \partial U_{2j}$. Therefore from Lemma  \ref{lemma:three_dom} 
we obtain either $g'(p)=0$ or $g(p)=p$, a contradiction either way.

If the orbit of $p$ is either periodic of period higher than $2$ or infinite (hence non-preperiodic) we denote by  $\orb(p)=\{p,g(p),g^2(p),\ldots\}$ and $\orb(q)=\{q=f^j(p), g(q), g^2(q),\ldots\}$  the orbits of $p$ and $f^j(p)$ under $g$, respectively, in  $\partial U_0 \cap \partial U_j$. 
By assumption their cardinality is at least $3$. If $\calo(p) = \calo(q)$,  applying $f^j$, we would have that all points in  $\orb(p)$
would be points in the common boundary  of $U_0, U_j$ and $ U_{2j}$, which is a contradiction to 
Lemma \ref{lem:three_domains_1}(b), since the cardinality of the orbits is at least three. 

Thus  $\orb(p)\ne \orb(q)$  should be two different orbits in $\partial U_j$.  Let $\eta=\{\eta_1,\eta_2,\ldots\}$ and $\tau=\{\tau_1,\tau_2,\ldots\}$ be the projection of the two orbits to the unit circle using the B\"ottcher coordinates of $U_j$.  Using the combinatorial result given by 
Lemma \ref{lemma:notmix}  we conclude than these two orbits are mixed, i.e. there exists four indexes $a, b, c$ and 
$d$ such that $\eta_a<\tau_b<\eta_c<\tau_d$. Applying Lemma \ref{lem:three_domains_1} (a)-(b), respectively, to the points 
$\hat{R}_0(\eta_a)$ and $\hat{R}_0(\eta_c)$ in $\partial U_0 \cap \partial U_j$ and $\hat{R}_0(\tau_c)$ and
$\hat{R}_0(\tau_d)$ in $\partial U_0 \cap \partial U_{2j}$ we arrive at a contradiction. 
 
From the arguments above the preperiodic case is also not possible. So, if $2j\neq n$ then the only possible case is $g(p)=p$.

\vglue 0.1truecm
\noindent {\bf Case 2 ($2j=n$)} 
\vglue 0.1truecm

For the {\it symmetric} case $2j=n$ the arguments above do not hold since $U_{2j} = U_0$.
However there are two main ingredients that provide a contradiction.

On the one hand if  we walk along the boundary of $U_0$ starting at $p$, say counterclockwise, and we  find the points on  the orbit of $p$ in a certain order, then when we walk clockwise along  the boundary of $U_j$ starting at $p$  we  should find the points of its orbit in the same order. 
On the other hand, the map $f^j:U_j \to U_0$ is 1-to-1, extends to the boundary of $U_j$, and it sends $\overline{U_j}$ to $\overline{U_0}$ preserving orientation, that is,  it sends the arc of $U_j$ joining clockwise (respectively counterclockwise)  $a, b \in \partial U_0 \cap \partial U_j$  to the corresponding arc of $U_0$ joining clockwise (respectively counterclockwise)  $f^j(a), f^j(b) \in \partial U_0 \cap \partial U_j$. 
The latter condition  follows  since $f^j$ is a  holomorphic map such that $f^j(u_j)=u_0$ and sends rays and equipotentials defined in $U_j$ to rays and equipotentials defined in $U_0$. The following arguments which finish the proof of the proposition are direct 
consequences of these two remarks.

As before let  $\orb(p)=\{p,f^{2j}(p),f^{4j}(p)\ldots \}$ and $\orb(q)=\{q=f^{j}(p),f^{3j}(p),f^{5j}(p)\ldots \}$ be the orbits of $p$ and $f^j(p)$ under $g=f^n$, respectively. Notice that 
for all $\ell\geq 1$ we have $f^{2 j \ell}=g^\ell$. We assume $p\neq f^{2j}(p)$ and whish to arrive to a contradiction.
We first show that $f^{j}(p)\ne \{p,f^{2j}(p)\}$. If $f^j(p)=p$ we apply $f^j$ to both sides and we get $f^{2j}(p)=f^j(p)=p$, a contradiction. 
If $f^j(p)=f^{2j}(p)$ it is easy to get $f^{4j}(p)=f^{2j}(p)$, or equivalently, $g(g(p))=g(p)$. This would imply that $p$ is prefixed under $g$, 
a contradiction. Summarizing we have that  $\{p,f^{j}(p),f^{2j}(p)\}$ are three different points in $\partial U_0 \cap \partial U_j$.

Take $p \in \partial U_0 \cap \partial U_j$ and denote by $\theta_0^\ell:=\varphi_0(f^{j\ell}(p))\in \mathbb S^1$ and 
$\theta_j^\ell:=\varphi_j(f^{j\ell}(p))\in \mathbb S^1$, $\ell \geq 0$,
the angles projected  by the B\"{o}ttcher coordinates of $U_0$ and $U_j$ respectively. 
 For $i\geq 0$ and $s\in \{0,j\}$ let $\gamma_{s_\pm}^{i,i+1}$ denote the arc in the unit circle going from $\theta_s^i$ to $\theta_s^{i+1}$ clockwise $(+)$ or counterclockwise $(-)$. 
Without loss of generality we assume that 
$\theta_j^{2} \in \gamma_{j_+}^{0,1}$. If  $\theta_j^{2} \in \gamma_{j_-}^{0,1}$, the arguments are similar. 
Consequently  $\theta_0^{2} \in \gamma_{0_-}^{0,1}$.  The image under  $(\varphi_0 \circ f^j \circ \varphi_j^{-1})$  of $\gamma_{j_+}^{0,1}$ is $\gamma_{0_+}^{1,2}$. So 
$\theta_0^3\in \gamma_{0_+}^{1,2}$ and therefore  $\theta_j^3\in \gamma_{j_-}^{1,2}$. The image under $(\varphi_0 \circ f^j \circ \varphi_j^{-1})$ of 
$\gamma_{j_-}^{1,2}$ is $\gamma_{0_-}^{2,3}$. So 
$\theta_0^4\in  \gamma_{0_-}^{2,3}$ and therefore  $\theta_j^4\in \gamma_{j_+}^{2,3}$. Applying successively this process it follows  that $\{\theta_0^0,\theta_0^2,\theta_0^4,\ldots\}$ is an infinite monotone sequence of 
points in $\gamma_{0_-}^{0,1}$. Since these points correspond to an orbit under the doubling map, their limit can only be a fixed point. But the only fixed point of $\Theta$ is $\theta=0$ which is repelling, a contradiction.
\end{proof}

We have studied, in the previous proposition,  the intersections between  the boundaries of the Fatou components of a (super) attracting cycle for types C and D quadratic rational maps. For type C maps there is just one  such cycle; so the proposition above  already covers all possible intersections among boundaries of Fatou components. Indeed, any Fatou domain eventually maps to the cycle and hyperbolicity implies there are no critical points in the Julia set, hence if there are no intersections among the boundaries of Fatou components of the superattracting cycle, there are no intersections whatsoever. 

For type D maps the situation is quite different since the above arguments only apply to both superattracting cycles separately, but do not to possible intersections among components of different cycles. In fact the parameter plane contains open sets of parameters (small Mandelbrot sets) for which these contacts occur. The next result deals with these cases.
\begin{Proposition}\label{prop:typeD}
Let  $f \in \rat$  be a hyperbolic map of type $D$ and let $ m \geq n \geq 3  $.  We denote by $U_0, U_1, \cdots, U_{n-1}$ and $V_0, V_1, \cdots, V_{m-1}$  the Fatou components which form the two immediate basins of attraction of the two cycles. Assume that  $\partial U_{i_1}  \cap \partial U_{i_2}= \emptyset$ for $0\leq i_1 < i_2 \leq n-1$ and $\partial V_{j_1}  \cap \partial V_{j_2}= \emptyset$ for $0 \leq j_1 < j_2 \leq m-1$.   If $p \in \partial U_i \cap \partial V_j$ then
$n \mid m$. 
\end{Proposition}

\begin{proof}

We can assume that $f$ is critically finite, i.e.,  $f \in \per_n(0) \cap \per_m(0)$ (see Theorem \ref{th:Rees}). We label the Fatou components so that $U_0$ and $V_0$ contain the two critical points of $f$.  Assume  $f(U_i)=U_{i+1 \, (\bmod n) } $ and $f(V_i)=V_{i+1\, (\bmod m)}$. Since $f$ is of type $D$ and $n,m \geq 3$ we know that every  Fatou component is a Jordan domain (see Corollary \ref{cor:useful}) and the dynamics of $f^n:\overline{U_i} \to \overline{U_i}$  and $f^m:\overline{V}_j \to \overline{V}_j$ is conformally conjugate to $z \to z^2$ on the closed unit disc $\overline{\bd}$. We also denote by $u_0, u_1, \cdots, u_{n-1}$ and $v_0, v_1, \cdots, v_{m-1}$ the two superattracting cycles.

We suppose that $n \nmid m$, or equivalently we assume that $nk \neq 0 \, (\bmod \,m)$, for all $1 \leq k \leq m-1$. Let $p$ be a point in $\partial U_{i} \cap \partial V_{j}$ then 
\[ 
\begin{array}{rl}
p \in & \partial U_i \cap \partial V_{Êj \, (\bmod m)} \\
f^n(p) \in & \partial U_i \cap \partial V_{j+n \, (\bmod m) } \\
f^{2n}(p) \in & \partial U_i \cap \partial V_{j+2n \, (\bmod m) }\\
\cdots & \cdots \\
f^{n(m-1)}(p) \in & \partial U_i \cap \partial V_{j+(m-1)n \, (\bmod m)}. 
\end{array}
\]
\noindent On the one hand we have that if  $ \ell_1,\ell_2 \leq m-1$ then  $V_{j+\ell_1n \, ( \bmod m)} \neq V_{j+\ell_2n \, ( \bmod m)}$  since $nk \neq 0 \,( \bmod \, m)$ and on the other hand $f^{\ell_1n}(p) \neq f^{\ell_2n}(p)$ since by assumption $\partial V_{j+\ell_1n \, ( \bmod m) } \cap \partial V_{j+\ell_2n \, ( \bmod m)} =\emptyset$. From these two facts we have that  $\partial U_i$ has non-empty intersection with  $\partial V_0, \partial V_1,  \cdots, \hbox{Êand }\partial V_{m-1}$.  The same happens for the rest of  $\partial U_{i}$ for $0 \leq i \leq n-1$. In summary every Fatou component $U_i$  intersects all Fatou components $V_j$. We denote by $z_{ij}$ a point in the common boundary of $\partial U_i \cap \partial V_j$. We build a domain $\Omega_0$ such that the boundary of $\Omega_0$ is formed by several internal rays. The first one joins $u_0$ and $u_1$ passing through $V_0$ in the following way:  we connect  $u_0$, $z_{00}$, $v_0$, $z_{10}$ and $u_1$ using internal rays. The second one joins   $u_0$ to $u_1$ passing trough $V_1$ in the same fashion. We construct another domain $\Omega_1$. In this case the boundary of $\Omega_1$ is formed by two curves, the first one joining $u_0$ and $u_1$ passing trough $V_1$ and the second one joining $u_0$ to $u_1$ passing trough $V_2$. These divide the Riemann sphere into three domains $\Omega_0$, $\Omega_1$ and the complement of $\overline{\Omega_0 \cup \Omega_1}$. Now $u_2$ must belong to one of these three regions. Therefore it cannot be accessed through  internal rays by the three marked points $v_0$, $v_1$ and $v_2$.
\end{proof}

\subsection{Proofs of Theorems A and  B}\label{sec: proof_A_B}

The proof of Theorem A is a direct consequence of the results above. 

\begin{proof}[Conclusion of the proof of Theorem A]
Statement (a) follows directly from Proposition \ref{prop:fixed_point_boundary} while statement (b) follows from Proposition \ref{prop:typeD}.
\end{proof}

We finish this section with the proof of Theorem B.

\begin{proof}[Proof of Theorem B]
If  $f\in \per_n(\lambda)$ is of type A then, from Lemma \ref{lemma:sierpins}, $f$ has an attracting fixed point (the only attracting cycle). Hence 
Theorem \ref{th:Milnor} implies that the Julia set is totally disconnected. This proves (a). 

Observe that statement (b) is trivial for $n=1$ and it is a particular case of \cite{matings} for $n=2$. Hence we assume $n\geq 3$ and $f\in \per_n(0)$, (see Theorem \ref{th:Rees}). If $f$ is of type $B$ the free critical point  must belong to $U_i$ for some $i\ne 0$. So, $f$ has no superattracting fixed points and therefore each $U_i$ is a Jordan domain (see Corollary \ref{cor:useful}).
Observe that    $f^n:\overline{U_i} \to \overline{U_i},\ i =0,\ldots, n-1$ is
a  degree 4 map conjugate to $z\mapsto z^4$. Consequently $f^{n}\,|\,{\partial
  U_i}$ is conjugate to $\theta \mapsto 4 \theta$ on the unit circle $\bs^1
=\br / \bz$. Since the map is critically finite, every internal ray in $U_i$
lands at a well--defined point on $\partial U_i$, $i=0,\ldots n-1$.  It
follows that there are three fixed points of $f^n$ on $\partial U_i$, namely
$\gamma_i(0)$, $\gamma_i(1/3)$ and $\gamma_i(2/3)$, $i=0,\ldots, n-1$. By
construction each of these points is fixed under $f^n$, and so they are
periodic points of period $d$ for $f$ with  $d|n$. If one of them is periodic
of period $d<n$ then such a point must belong to $\partial U_i \cap \partial
U_j$ for some $i\neq j$ and so $\J$ cannot be a Sierpi\'{n}ski curve. So,
let us assume $d=n$ (the only case compatible with $\J$ being a Sierpi\'nski curve), 
and show that this is not possible if $n=3$ or $n=4$. 

If $d=n$, the $3n$ points involved in the construction form $3$ different cycles of period $n$ for $f$.  So $f$ would have,  globally, at least  $4$ cycles of period $n$ since each $f\in \per_n(0)$ has one (further) superattracting $n$ cycle. However a quadratic rational map has at most  $2$  cycles of period $n=3$, and  $3$ cycles of period  $n=4$ respectively, a contradiction. 

From Lemma \ref{lemma:sierpins} statement (c) reduces to consider the possible contact points among boundaries of Fatou components. From Theorem A(a), we immediately conclude that there are no contacts among the Fatou components of the unique attracting cycle of $f$. Finally, if $f$ is of type $C$, then any other Fatou component is a preimage of one of the components of the attracting cycle and since $f$ is hyperbolic there are no critical points in the Julia set. So, there are no possible contacts among boundaries whatsoever.

The proof of statement (d) follows immediately  from Theorem A(b).
\end{proof}

\section{The period three  slice. Proof of Theorem C }\label{sec:slice3}

In this section we restrict our attention to rational maps in $Per_3(0)$. This slice contains all the conformal conjugacy classes of maps in $\mathcal{M}_2$  with a periodic critical orbit of period three. Using a  suitable M\"obius transformation we can assume that one critical point is located at the origin, and the critical cycle is $0 \mapsto \infty \mapsto 1 \mapsto 0$. Such maps can
be written as  $(z-1)(z-a)/z^2$, and using this expression the other critical point is now located at $2a/ (a+1)$. We may change this parametrization of $Per_3(0)$ so that the critical point is  located at $a$, obtaining the following expression 
\begin{equation}\label{eq:function}
f_a(z)=\frac{(z-1)\left(z-\frac{a}{2-a} \right)}{z^2} \qquad \hbox{ where } a \in \bc \setminus \{ 0,2\}.
\end{equation}

We exclude the values $a=0$ ($f_0$ has degree 1) and $a=2$ ($f_2$ is not well defined).
As we mentioned before, $f_a$, for $a \in \bc \setminus \{0,2\}$, has a superattracting cycle $0 \mapsto \infty \mapsto 1 \mapsto 0$ and we denote by $U_0=U_0(a), U_{\infty}=U_{\infty}(a), U_{1}=U_{1}(a)$ the Fatou components containing the corresponding points of this superattracting cycle.  This map has two critical points, located at $c_1=0$ and $c_2(a)=a$, and the corresponding critical values are $v_1=\infty$ and $v_2(a)=-\frac{(1-a)^2}{a(2-a)}$.  Thus, the dynamical behaviour of  $f_a$ is determined by the orbit of the free critical point $c_2(a)=a$. The parameter $a-$plane has been thoroughly studied by M. Rees (\cite{Re3}) and we recall briefly some of its main properties. We parametrize the hyperbolic components of $Per_3(0)$ by the unit disc in the natural way.  For the Bitransitive and Capture components we use the well defined B\"ottcher map in a small neighbourhood of  each point of the critical cycle $\{0,\infty,1\}$ and for the Disjoint type components the multiplier of the attracting cycle different from $\{0,\infty,1\}$.

The first known result is the existence of  only two Bitranstitve components (\cite{Re3}) denoted by $B_1$ and $B_{\infty}$ and defined by 
\[
B_1 = \{ a \in \bc \, | \, a \in U_1(a)\} \, \hbox{ and } B_{\infty} = \{ a \in \bc \, | \, a \in U_{\infty}(a)\}.
\]

\noindent  $B_1$ is open, bounded, connected and simply connected and $B_{\infty}$ is open, unbounded, connected and simply connected in $\hat\bbC$.  In the next result we collect these and other  main known properties of the parameter plane (see Figures \ref{fig:slice3_dos} and \ref{fig:slice3_poly_rat}). 

\begin{Proposition}[\bf{Rees, \cite{Re3}}]  \label{Pro:Rees}
For $f_a(z)$ with $a \in \bc \setminus \{0,2\}$, the following conditions hold:
\begin{enumerate}
\item[(a)] The boundaries of $B_1$ and $B_{\infty}$ meet at three parameters $0, x$ and $ \bar{x}$ and the set $\bc \setminus (B_1 \cup B_{\infty} \cup \{0,x,\bar{x}\} ) $ has exactly three connected components: $\Omega_1, \Omega_2$ and $\Omega_3$.
\item[(b)] Each connected component $\Omega_i$, for $i=1,2,3$, contains a unique value $a_i$ such that 
$f_{a_i}$ is conformally conjugate to a polynomial map of degree 2. Moreover, each one of the three parameters $a_i$ is the center of a hyperbolic component $\Delta_i$ of period one. 
\item[(c)] Each parameter value, $0, x$ and $\bar{x}$, is the landing point of two fixed parameter rays, one  in $B_1$ and one in $B_{\infty}$.
\item[(d)] The parameter values $x$ and $\bar{x}$ correspond to parabolic maps having a fixed point with multiplier $e^{2 \pi i /3}$ and $e^{-2 \pi i /3}$, respectively.
\end{enumerate}
\end{Proposition}

In Figure \ref{fig:slice3_dos} we plot the $a-$parameter plane. In this picture we label the two hyperbolic components $B_1$ and $B_{\infty}$ of  {\it Bitransitive} type and the cutting points $0,x $ and $\bar{x}$ that separate this parameter plane into three different zones: $\Omega_1, \Omega_2$ and $\Omega_3$. Each zone contains a unique parameter $a$ such that $f_a$  is conformally conjugate to a quadratic polynomial. We will show that  these three parameter values are $a_1, a_2$ and $\overline{a_2}$ (plotted with a small black circle), which correspond to the airplane, the rabbit and the co-rabbit, respectively. 

\begin{figure}[ht]
	\centering
	\includegraphics[width=0.5\textwidth]{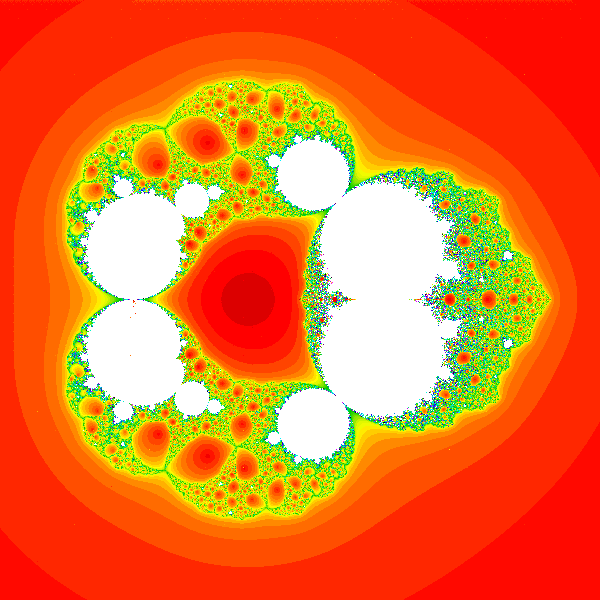}
	\setlength{\unitlength}{220pt}
         \put(-0.75,0.6){\circle*{0.015}}
         \put(-0.75,0.4){\circle*{0.015}}
         \put(-0.35,0.5){\circle*{0.015}}
         \put(-0.83,0.55){\small $\Delta_2$}
         \put(-0.83,0.43){\small $\Delta_3$}
         \put(-0.35,0.61){\small $\Delta_1$}
         \put(-0.35,0.475){\tiny $a_1$}
         \put(-0.77,0.37){\tiny $a_3$}
         \put(-0.77,0.62){\tiny $a_2$}
	\put(-0.6,0.5){\small $B_1$}
	\put(-0.94,0.85){\small $B_\infty$}
	\put(-0.65,0.78){\small $\bf\Omega_2$}
	\put(-0.65,0.25){\small $\bf\Omega_3$}
	\put(-0.2,0.5){\small $\bf \Omega_1$}
	\put(-0.90,0.5){\vector(1,0){0.08}}
	\put(-0.95,0.49){\small $0$}
	\put(-0.39,0.72){\vector(1,1){0.08}}
	\put(-0.27,0.81){\small $x$}
         \put(-0.38,0.28){\vector(1,-1){0.08}}
	\put(-0.27,0.18){\small $\bar{x}$}
	\put(-0.85,0.05){\footnotesize $\Omega_i=\bc \setminus \left( B_1\cup B_\infty \cup \{ 0,x,\bar{x} \}  \right)$}	
	  \caption{\small{The slice $Per_3(0)$.}} 
	\label{fig:slice3_dos}
\end{figure}

\begin{figure}[hbt]
    \centering
    \subfigure[\scriptsize{The Douady rabbit. The Julia set  of $z^2-0.122561+0.744861i$.}  ]{
     \includegraphics[width=0.32\textwidth]{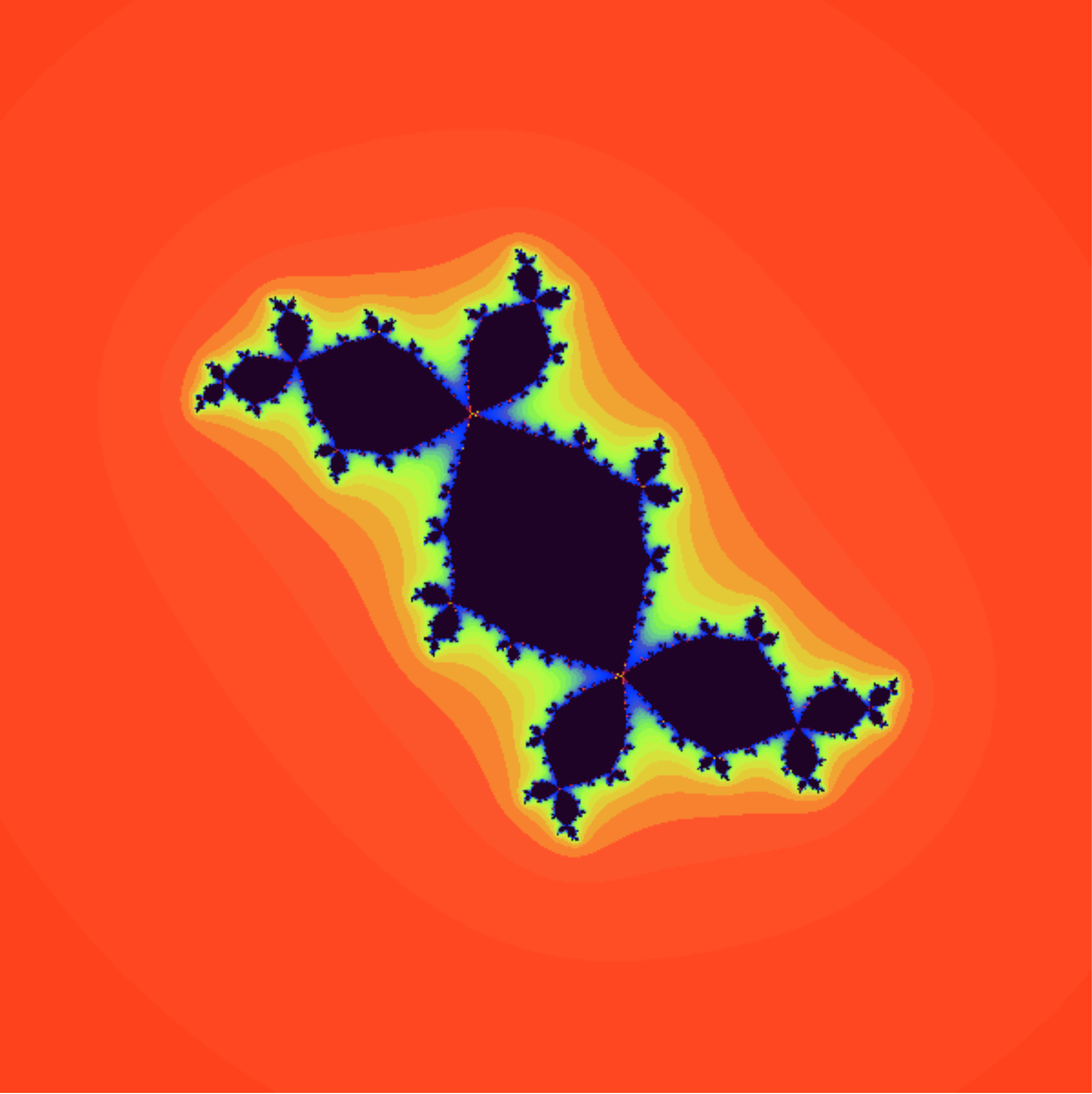}}
    \subfigure[\scriptsize{The Douady co--rabbit. The Julia set  of $z^2-0.122561-0.744861i$.}  ]{
        \includegraphics[width=0.32\textwidth]{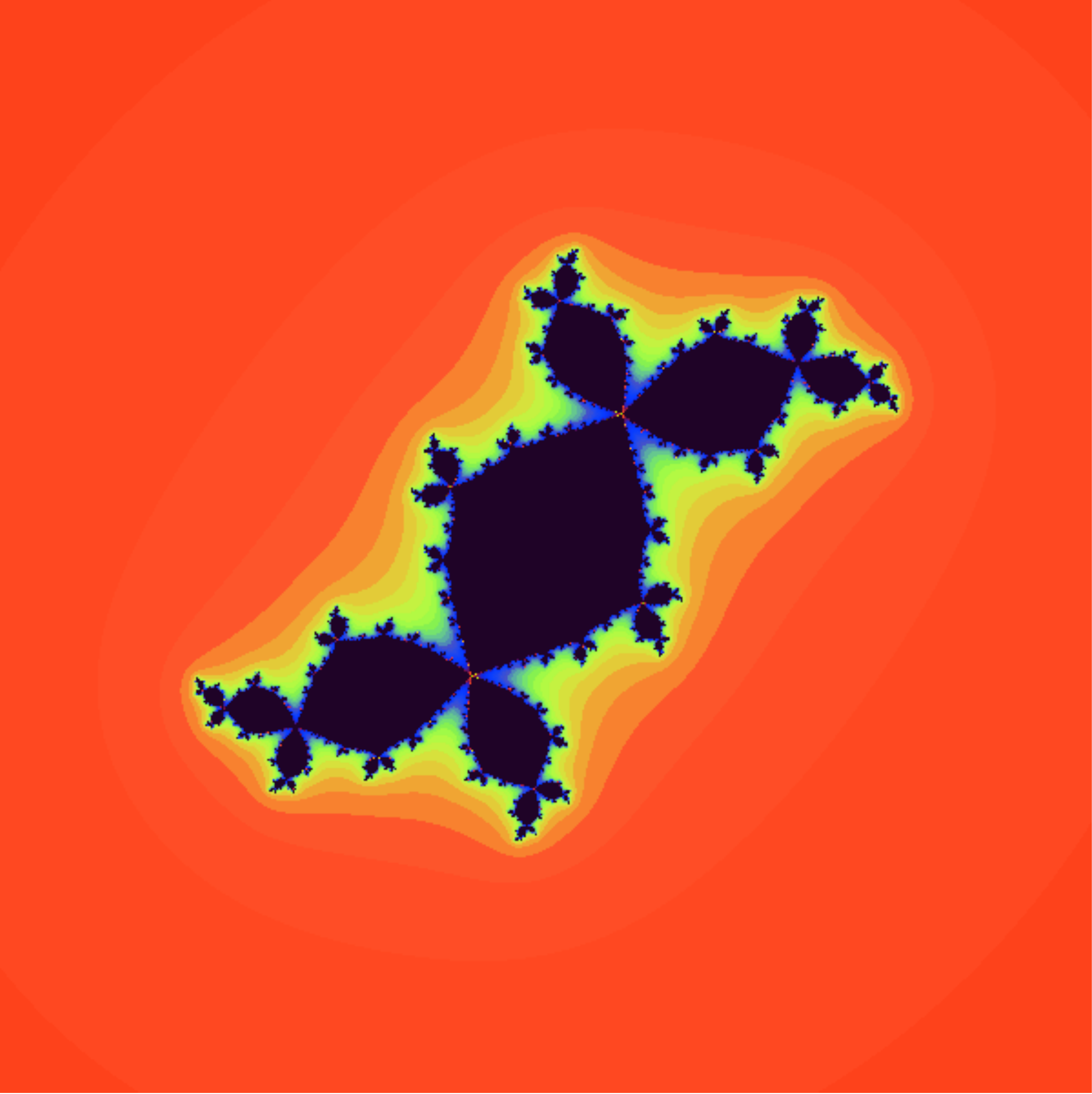}}
    \subfigure[\scriptsize{The airplane. The Julia set of $z^2-1.75488$.}  ]{
     \includegraphics[width=0.32\textwidth]{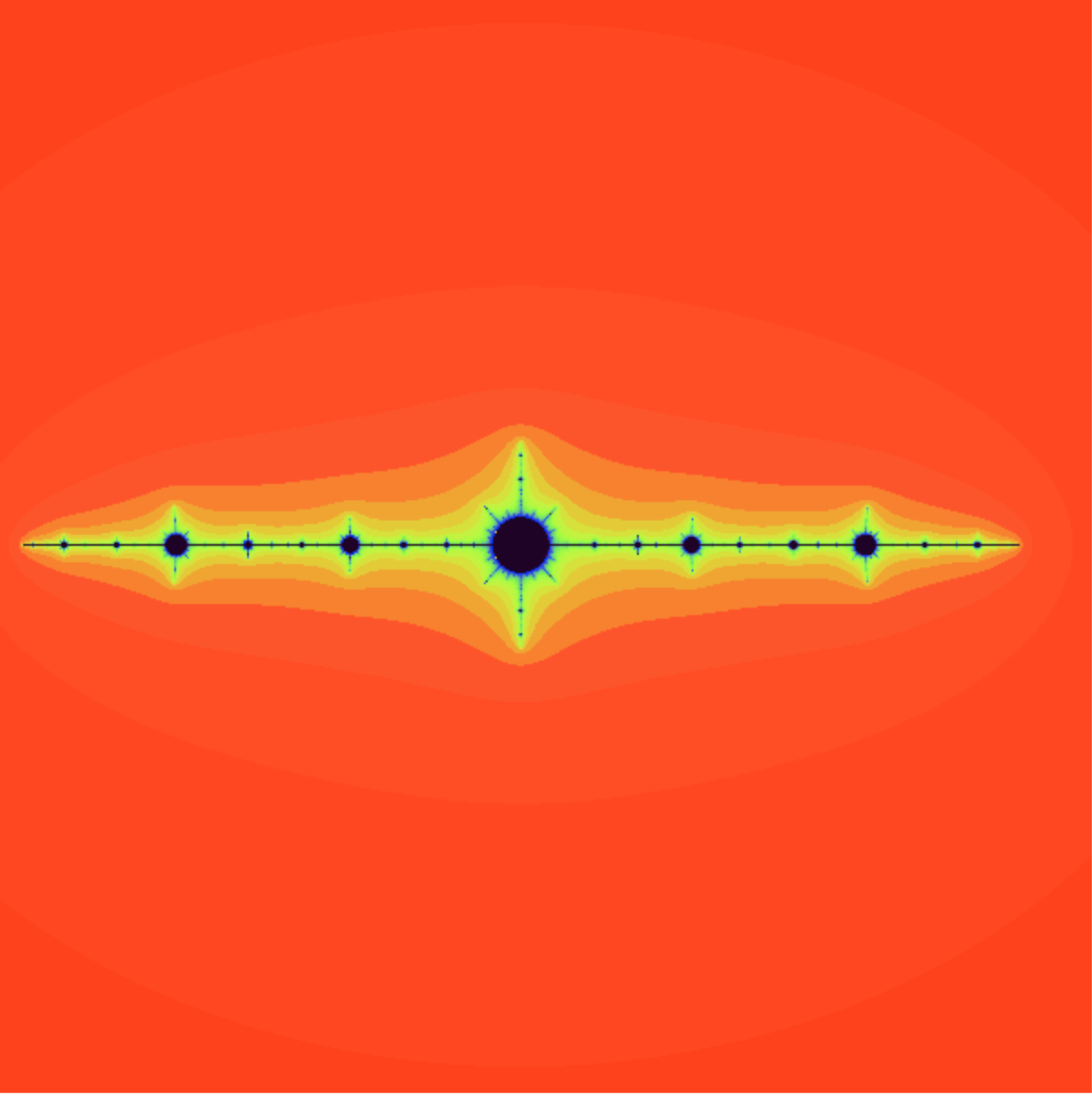}}
    \subfigure[\scriptsize{Julia set of  $f_{0.33764+0.56228i}$, in $Per_3(0)$,   conjugate to the Douady rabbit.}  ]{
    \includegraphics[width=0.32\textwidth]{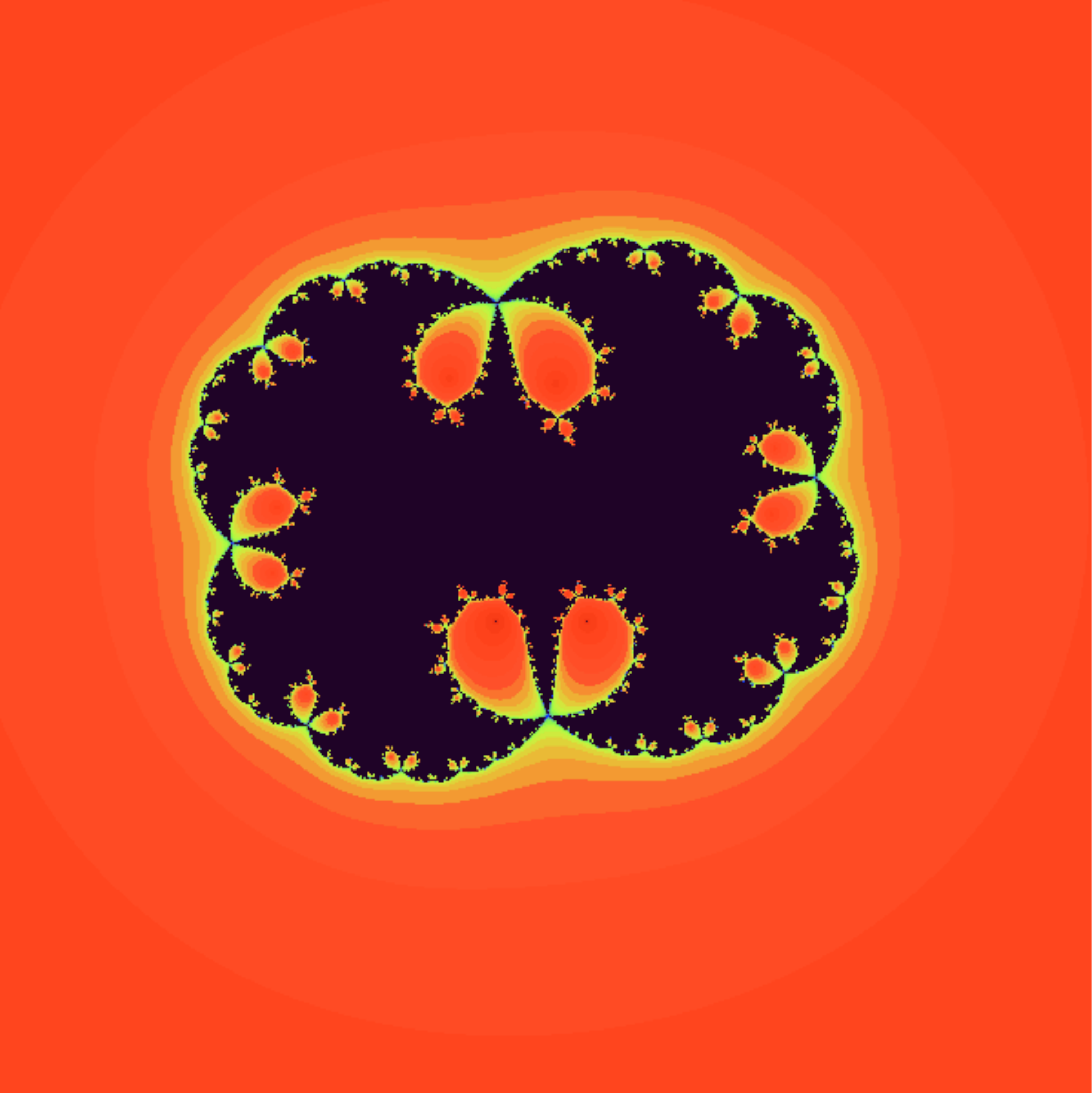}}
    \subfigure[\scriptsize{Julia set of $f_{0.33764-0.56228i}$, in $Per_3(0)$, conjugate to the Douady co-rabbit.}  ]{
        \includegraphics[width=0.32\textwidth]{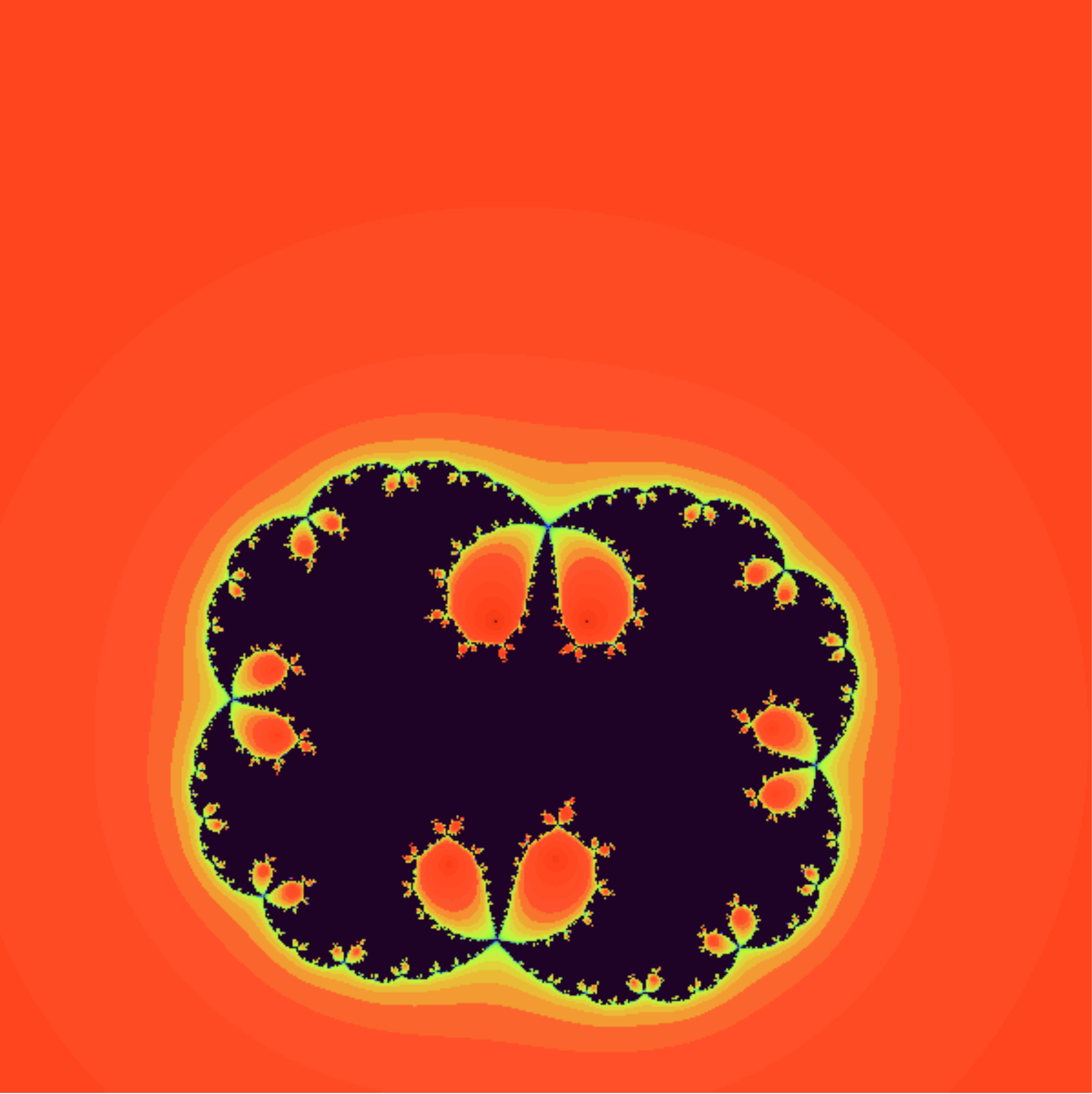}}
    \subfigure[\scriptsize{Julia set of $f_{2.32472}$, in $Per_3(0)$, conjugate to the airplane.}  ]{
   \includegraphics[width=0.32\textwidth]{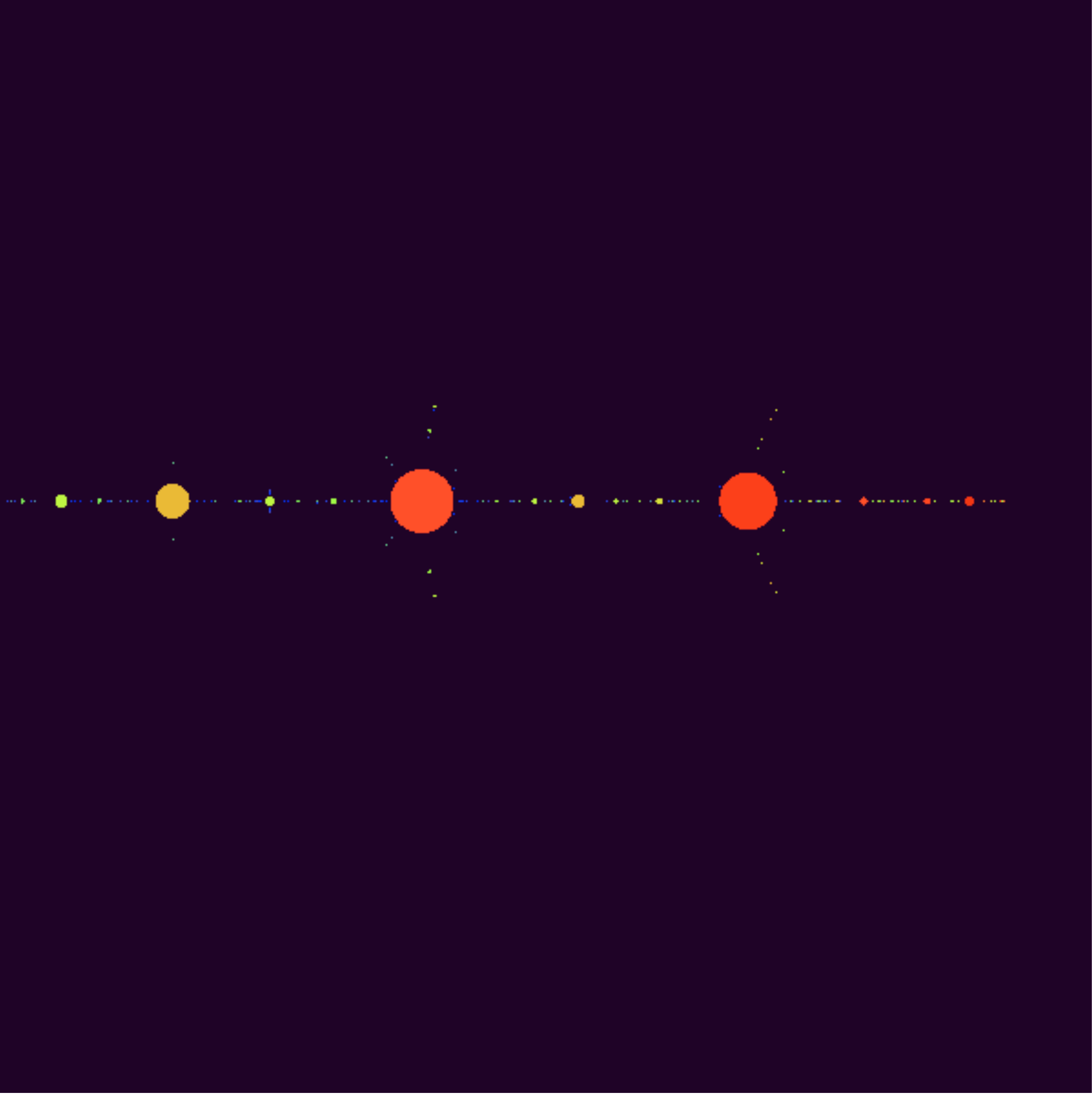}}
    \caption{\small{ We plot the three unique monic, quadratic,  centered polynomial  having a superattracting 3-cycle:  the  rabbit,  the co-rabbit and the airplane, and the three corresponding rational maps $f_a$ that are conformally conjugate to a quadratic polynomial.}}    
    \label{fig:slice3_poly_rat}
\end{figure}

We can find explicitly   the values  of $x$ and $\bar{x}$ and the quadratic polynomial $f_{a_i}$, for $i=1,2, 3$. 
First, we calculate  the three parameters $a_1, a_2$ and $a_3$ such that the corresponding quadratic rational map $f_{a_i}$ is conformally conjugate to a quadratic polynomial. This can happen if and only if the free critical point $c_2(a)=a$ is a superattracting fixed point. This superattracting fixed point plays the role of $\infty$ for the quadratic polynomial.  This condition says that the corresponding critical value $v_2(a)$ coincides with the critical point $c_2(a)$, or equivalently
\[
v_2(a)=-\frac{(1-a)^2}{a(2-a)} = a
\]
\noindent which yields
\[
a^3-3a^2+2a-1=0.
\]
The above equation has one real solution $a_1 \approx 2.32472$ and two complex conjugate solutions $a_2 \approx 0.33764+0.56228i $ and $a_3\approx 0.33764-0.56228i$. Notice that there are only three monic and centered quadratic polynomials of the form $z^2 + c$ with a 3-critical cycle. These three polynomials are the airplane $z^2-1.7588$, the rabbit 
$z^2  -0.122561+0.744861i$ and the co-rabbit $z^2  -0.122561-0.744861i$. We claim that $f_{a_1}$  is conformally conjugate  to the airplane, $f_{a_2}$ to the rabbit and $f_{\overline{a_2}}$ to the co-rabbit. To see this we define the map 
\[
\tau(z)=\frac{ 1 }{z-a_i} + \frac{1}{a_i}
\]
\noindent and then $P_i:=\tau \circ f_{a_i} \circ \tau^{-1}$ is a centered  quadratic polynomial, since $\infty$ is a superattracing fixed point and $z=0$ is the unique finite critical point. Easy computations show that
\[
P_i(z)=\frac{1}{a_i}-a_i^3(a_i-2)z^2.
\]
Finally, after conjugation with the affine map $\sigma(z)=-a_i^3(a_i-2)z$, the corresponding quadratic polynomial $Q_i:=\sigma \circ P_i \circ \sigma^{-1}$ is given by 
\[
Q_i(z)=z^2-a_i^2(a_i-2),
\]
\noindent which coincides with the airplane for $i=1$, the rabbit for $i=2$
and the co-rabbit for $i=3$.  We call $a_1$ the {\it airplane parameter}, $a_2$ the {\it rabbit parameter} and $\overline{a_2}$ the {\it co-rabbit parameter}.  Likewise, we call $\Omega_1$  the {\it airplane piece} since it contains the airplane parameter $a_1$,  $\Omega_2$ the {\it rabbit piece} since it contains the rabbit parameter and $\Omega_3$ the {\it co-rabbit piece} since it contains the  co-rabbit parameter.

In the next proposition we show another  property of the cutting parameter values $x$ and $\bar{x}$, that will be important in order 
to determine their values.
 
 \begin{Proposition}\label{lem:delta1}
Let $\Delta_i$ be the hyperbolic  component containing $a_i$ (so that $\Delta_i\subset \Omega_i$), $i=1,2,3$. Then, the cutting parameter values $x$ and $\bar{x}$ in Proposition \ref{Pro:Rees} belong to the boundary of $\Delta_1$,  and not to  the boundary of $\Delta_2$ and $\Delta_3$.
 \end{Proposition}

\begin{proof}
When a parameter $a$ belongs to any of the $\Delta_i, \ i=1,2,3$, the corresponding dynamical plane exhibits a fixed basin of an attracting fixed  point denoted, in what follows, by $p(a)$. From Proposition \ref{Pro:Rees} we know that  $f_{x}$ (respectively $f_{\bar{x}}$) has a parabolic fixed point, $p(x)$ (respectively $p(\bar{x})$), with multiplier $e^{2 \pi i /3}$ (respectively $e^{-2 \pi i /3}$). Thus $x$ and $\bar{x}$ must belong to 
  $\partial \Delta_1$, $\partial \Delta_2$, or $\partial \Delta_3$. Moreover since $x$ and $\bar{x}$ also belong to  $\partial  B_1$ (and $\partial  B_{\infty}$), the dynamical planes for $f_{x}$ and $f_{\bar{x}}$ are such that  $p(x)$ and $p(\bar{x})$  belong to $\partial U_0 \cap \partial  U_1  \cap \partial U_{\infty}$. These are the two conditions defining the parameters $x$ and $\bar{x}$ (see Figure \ref{fig:slice3_dos}). 
  
When the parameter $a$,  belonging to any of the $\Delta_i, \ i=1,2,3$, crosses the  boundary of its hyperbolic component  through its $1/3$--bifurcation point, the attracting fixed point $p(a)$ becomes a parabolic fixed point  of multiplier either $e^{2 \pi i /3}$ or $e^{-2 \pi i /3}$ since, at this precise parameter value, the attracting fixed point coalesces with a repelling periodic orbit of period three. 

Since $f_a, \ a\in\bc$, is a quadratic rational map, it has only two 3-cycles and, because we are in $Per_3(0)$, one of them is the critical cycle  $\{0,\infty,1\}$. So, the repelling periodic orbit which coalesces with $p(a)$ at the 1/3--bifurcation parameter  must be the unique repelling 3-cycle existing for this parameter. 

We now investigate the location of this repelling 3-cycle for parameters in each of the hyperbolic components  $\Delta_1, \Delta_2$ and $\Delta_3$. To do so, we note that if $a$ is any parameter
in $\Delta_i$,  we have that $f_a^3: \overline{U_0} \mapsto \overline{U_0}$ is conjugate to the map $z \mapsto z^2$ in the closed  unit disc. Thus, there exists a unique point $z_0(a)\in \partial U_0$ such that $f_a^3\left(z_0(a)\right)=z_0(a)$. This fixed point could be either  a (repelling) fixed point for $f_a$ or a (repelling) 3-cycle of  $f_a$.

It is clear that for $a=a_1$ the point  $z_0\left(a_1\right)$ is a repelling 3-cycle, since, for the airplane,  $\partial U_0 \cap \partial U_\infty \cap \partial U_1$ is empty. So, this configuration remains for all parameters in $\Delta_1$ (the hyperbolic component  containing the airplane parameter). At the 
1/3--bifurcation points of $\Delta_1$, the repelling periodic orbit $\{z_0(a),f\left(z_0(a)\right),f^2\left(z_0(a)\right) \}$) coalesces with $p(a)$ (the attracting fixed point), and this collision must happen  in $\partial U_0 \cap \partial U_\infty \cap \partial U_1$.  So the 1/3--bifurcation parameters of $\Delta_1$ are precisely the parameter values $a=x$ and $a=\bar{x}$, and so, $p(a)$ becomes either $p(x)$ or $p(\bar{x})$, respectively.

On the other hand for $a=a_i,\ i=2,3$ the point  $z_0\left(a_i\right)$ is a fixed point (since for the rabbit and co-rabbit  $\partial U_0 \cap \partial U_\infty \cap \partial U_1$ is precisely  $z_0\left(a_i\right)$). As before this configuration remains for all parameters in $\Delta_i, \ i=2,3$ (the hyperbolic components  containing the rabbit and co-rabbit, respectively). Therefore,  at the 1/3--bifurcation point of $\Delta_i,\ i=2,3$, the fixed point $p(a)$ coalesces with the repelling periodic orbit but this collision does not happen  in  $\partial U_0 \cup \partial U_\infty \cup \partial U_{1}$ since the repelling periodic orbit of period three does not belong to $\partial U_0 \cup \partial U_\infty \cup \partial U_{1}$. Consequently the resulting parabolic point is not in $\partial U_0 \cap \partial U_\infty \cap \partial U_{1}$ and the 1/3--bifurcation parameter can neither be $x$ nor $\bar{x}$.
\end{proof}

Doing easy numerical computations we get that there are five parameter values having a parabolic fixed point with multiplier  $e^{2 \pi i /3}$ or $e^{-2 \pi i /3}$. These are
\[
0, \quad  1.84445\ldots \pm i 0.893455\ldots,   \quad   0.441264\ldots \pm i 0.59116\ldots.    
\]
It is easy to show that $x \approx 1.84445 + 0.893455i$ (and so, $\bar{x} \approx 1.84445 - 0.893455i$). Thus the parameters $0.441264 \pm 0.59116i$ corresponds to the 1/3--bifurcations of $\Delta_2$ and $\Delta_3$, respectively.

Now we are ready to  prove Theorem C.

\begin{proof}[Proof of Theorem C]\

\begin{itemize}
\item[(a)] Assume $a\in \left( B_1 \cup B_\infty \right)$. From Theorem \ref{th:Rees} we know that $B_1$ has a unique
center at $a=1$. Likewise,  $a=\infty$ is the unique center of $B_\infty$.  In either case the corresponding map $f_{a_0}$ is a critically finite hyperbolic map in $Per_3(0)$ of type $B$. Thus, from Theorem B (b) $\mathcal J\left(f_{a_0}\right)$ is not a Sierpi\'{n}ski curve. We conclude that $\mathcal J\left(f_a\right)$ is not a Sierpi\'{n}ski curve either, since all Julia sets 
in the same hyperbolic component are homeomorphic.

\item[(b)] Assume $a\in \Omega_2$ (here we do not restrict to a hyperbolic
  parameter). From the previous proposition we know that there exists a fixed
  point $z_0(a)\in \partial U_0 \cap \partial U_\infty \cap \partial U_{1}$ and this fixed point is the natural continuation of  $z_0\left(a_2\right)$ which cannot bifurcate until $a=x \in \Delta_1$. The case $a \in \Omega_3$ is similar.

\item [(c)]Finally we  assume $a \in \mathcal H$ where $\mathcal H$ is a hyperbolic component of type $C$ in  $\Omega_1$. We know that $\Omega_1$ contains the airplane polynomial for which $\partial U_0 \cap \partial U_\infty \cap \partial U_{1}=\emptyset$. This configuration cannot change unless the period 3 repelling cycle coalesces with a fixed point, which 
only happens at $a=x$ or $a=\bar{x}$. Hence the intersection is empty for all parameters in $\Omega_1$. It follows from
Theorem B (c) that this is the precise condition for $\mathcal J(f_a)$ to be a Sierpi\'{n}ski curve.

\item[(d)] This case a direct application of Theorem B (d).
\end{itemize}
\end{proof}

\bibliography{sierp}

\begin{thebibliography}{10}

\bibitem{matings}
Magnus Aspenberg and Michael Yampolsky.
\newblock Mating non-renormalizable quadratic polynomials.
\newblock {\em Comm. Math. Phys.}, 287(1):1--40, 2009.

\bibitem{escapet}
Robert~L. Devaney, Daniel~M. Look, and David Uminsky.
\newblock The escape trichotomy for singularly perturbed rational maps.
\newblock {\em Indiana Univ. Math. J.}, 54(6):1621--1634, 2005.

\bibitem{GarJarMor}
Antonio Garijo, Xavier Jarque, and M{{\'o}}nica Moreno~Rocha.
\newblock Non-landing hairs in {S}ierpi{\'n}ski curve {J}ulia sets of
  transcendental entire maps.
\newblock {\em Fund. Math.}, 214(2):135--160, 2011.

\bibitem{mss}
Ricardo Ma{\~n}{{\'e}}, Paulo Sad, and Dennis Sullivan.
\newblock On the dynamics of rational maps.
\newblock {\em Ann. Sci. {\'E}cole Norm. Sup. (4)}, 16(2):193--217, 1983.

\bibitem{Mil}
John Milnor.
\newblock Remarks on iterated cubic maps.
\newblock {\em Experiment. Math.}, 1(1):5--24, 1992.

\bibitem{MilLei}
John Milnor.
\newblock Geometry and dynamics of quadratic rational maps.
\newblock {\em Experiment. Math.}, 2(1):37--83, 1993.
\newblock With an appendix by the author and Lei Tan.

\bibitem{milnor}
John Milnor.
\newblock {\em Dynamics in one complex variable}, volume 160 of {\em Annals of
  Mathematics Studies}.
\newblock Princeton University Press, Princeton, NJ, third edition, 2006.

\bibitem{Moro}
Shunsuke Morosawa.
\newblock Local connectedness of {J}ulia sets for transcendental entire
  functions.
\newblock In {\em Nonlinear analysis and convex analysis ({N}iigata, 1998)},
  pages 266--273. World Sci. Publ., River Edge, NJ, 1999.

\bibitem{Pil}
Kevin~M. Pilgrim.
\newblock Rational maps whose {F}atou components are {J}ordan domains.
\newblock {\em Ergodic Theory Dynam. Systems}, 16(6):1323--1343, 1996.

\bibitem{Re0}
Mary Rees.
\newblock Components of degree two hyperbolic rational maps.
\newblock {\em Invent. Math.}, 100(2):357--382, 1990.

\bibitem{Re1}
Mary Rees.
\newblock A partial description of parameter space of rational maps of degree
  two. {I}.
\newblock {\em Acta Math.}, 168(1-2):11--87, 1992.

\bibitem{Re2}
Mary Rees.
\newblock A partial description of the parameter space of rational maps of
  degree two. {II}.
\newblock {\em Proc. London Math. Soc. (3)}, 70(3):644--690, 1995.

\bibitem{Re3}
Mary Rees.
\newblock A fundamental domain for {$V_3$}.
\newblock {\em M{\'e}m. Soc. Math. Fr. (N.S.)}, (121):ii+139, 2010.

\bibitem{Stein}
Norbert Steinmetz.
\newblock Sierpi{\'n}ski and non-{S}ierpi{\'n}ski curve {J}ulia sets in
  families of rational maps.
\newblock {\em J. Lond. Math. Soc. (2)}, 78(2):290--304, 2008.

\bibitem{Why}
G.~T. Whyburn.
\newblock Topological characterization of the {S}ierpi{\'n}ski curve.
\newblock {\em Fund. Math.}, 45:320--324, 1958.

\end{thebibliography}

\end{document}